\documentclass[12pt,reqno]{amsart}
             
\usepackage[left=1in, right=1in, top=1.1in,bottom=1.1in]{geometry} 
\usepackage[dvipsnames]{xcolor}
\usepackage[mathscr]{eucal}
\usepackage{latexsym,amsfonts,amssymb}
\usepackage{bm,pifont,stmaryrd}
\usepackage{graphicx}
\usepackage{color}
\usepackage{enumitem}
\usepackage{hyperref}
\hypersetup{
     colorlinks   = true,
     citecolor    = blue,
     linkcolor    = blue
} 

\usepackage{pdfsync}
\usepackage[font={scriptsize}]{caption}

\usepackage{tikz}
\usetikzlibrary{trees}

\definecolor{darkblue}{rgb}{0.1, 0.2, 0.75}
\definecolor{darkgreen}{rgb}{0.1, 0.35, 0}
\definecolor{candypink}{rgb}{0.89, 0.44, 0.48}
\definecolor{deepcerise}{rgb}{0.85, 0.2, 0.53}
\definecolor{atomictangerine}{rgb}{1.0, 0.6, 0.4}
\definecolor{fandango}{rgb}{0.71, 0.2, 0.54}

\newtheorem{theorem}{Theorem}[section]
\newtheorem{Def}[theorem]{Definition}

\newtheorem{thm}[theorem]{Theorem}

\newtheorem{lemma}[theorem]{Lemma}

\theoremstyle{remark}
\newtheorem{remark}[theorem]{Remark}
\newtheorem{hyp}[theorem]{Hypothesis}
\numberwithin{equation}{section}

\def\RR{\mathbb{R}}

\def\NN{\mathbb{N}}
\def\mE{\mathbb{E}}

\newcommand{\DD}{\mathbb{D}}

\def\bfx{{\bf x}}

\newcommand{\ca}{{\mathcal A}}

\newcommand{\cf}{{\mathcal F}}

\newcommand{\ch}{{\mathcal H}}
\newcommand{\ci}{{\mathcal I}}
\newcommand{\cj}{{\mathcal J}}

\newcommand{\cm}{{\mathcal M}}
\newcommand{\cn}{{\mathcal N}}

\newcommand{\cp}{{\mathcal P}}

\newcommand{\cv}{{\mathcal V}}

\def\la{{\lambda}}
\def\si{\sigma}

\def\la{{\lambda}}

\def\al{{\alpha}}

\def\be{{\beta}}

 \newcommand{\ep}{\varepsilon}
 
\newcommand{\lp}{\left(}
\newcommand{\rp}{\right)}
\newcommand{\lc}{\left[}
\newcommand{\rc}{\right]}

\begin{document}
\title[Riemann-Skorohod integrals]
{Riemann-Skorohod and Stratonovich integrals  for Gaussian processes}
\date{}   

\author[Y. Liu]
{Yanghui Liu} 
\address{Y. Liu: Baruch College, CUNY, New York}
\email{yanghui.liu@baruch.cuny.edu}


    \begin{abstract}
In this paper we consider    Skorohod and Stratonovich-type integrals   in a general setting of Gaussian processes. 
We show that a conversion formula holds when  the covariance functions  of the Gaussian process are of finite $\rho$-variation for $\rho\geq 1$ and    that  the  diagonals  of   covariance functions   are of finite $\rho'$-variation for $\rho'\geq 1$ such that $\frac{1}{\rho'}+\frac{1}{2\rho}>1$. 
 The difference between the two types of integrals is identified  with  a Young integral.  
We also show that    the   Skorohod  integral is the limit of a $[\rho]$-th order Skorohod-Riemann sum.   
\end{abstract}

\maketitle


\section{Introduction}
  
  

  




Let $x$ be a continuous   semimartingale   and $f $  a   sufficiently smooth function.  
The classical It\^o-Stratonovich conversion  formula states that the difference between    Stratonovich integral $\int_{0}^{T}f(x_{t})  dx_{t}$ and      It\^o integral, denoted by  $\int_{0}^{T}f(x_{t})\delta x_{t}$\,,  is equal to the Riemann-Stieltjes  integral $\frac12 \int_{0}^{T}\Delta f(x_{t})d[x]_{t}$,  where $\Delta f  $ is the Laplacian of $f$ and $[x]$ is the quadratic variation process of $x$ (see e.g. \cite{karatzas1991brownian, revuz2013continuous}). Furthermore, if the integrand $f(x_{\cdot})$ is replaced by a continuous  semimartingale $y$ then the difference is identified with the covariation $\frac12 [y,x]_{T}$.

Interestingly,   both types of stochastic calculus have been extended to the setting of Gaussian processes (by  entirely different approaches). The theory of rough path  provides a path-wise approach for Gaussian processes (see e.g. \cite{FV10}), which gives rise to   Stratonovich-type integrals, while  the Skorohod integrals   based on Malliavin calculus  generalize the It\^o integrals (see e.g. \cite{hu2016analysis, N06}). 
This raises a natural question: Is there a conversion formula for Gaussian processes as well that   relates these two approaches? If so,   to what extent  does it hold?

This problem has been investigated in a number of recent works  (see  e.g.   \cite{alos2001stochastic, carmona2003stochastic, cass2019stratonovich,  cass2021skorohod, coutin2007introduction, hu2013stratonovich,  kruk2010malliavin, nualart2011multidimensional,   song2022skorohod}).  
The      results    are  mainly focused on   a  one-dimensional setting, or   assume that a geometric  Gaussian  rough path     exists.  The latter usually requires that the covariance function of the Gaussian process is of finite $\rho$-variation for $\rho<2$. 
These conditions are in place  to   properly define the Stratonovich-type integral. 
Another main obstacle for the conversion formula problem is that the Skorohod integral  $\int_{0}^{T}f(x_{t})\delta x_{t}$ does not exist when $x$ is, for instance, a fractional Brownian motion with Hurst parameter $H\leq 1/4$. This problem is resolved by defining a proper extension of the divergence operator (see e.g. \cite{cheridito2005stochastic}). In \cite{hu2013stratonovich}, the authors also studied this extension  based on Wick product.   
  
In this   paper, we consider the  conversion formula problem in a more general   setting of Gaussian processes. More precisely, we will assume the following:

\begin{hyp}\label{hyp.1}
Let $ x=(x_{1},\dots, x_{d} )$ be a $\RR^{d}$-valued mean-zero continuous  Gaussian process such that  the components $x_{i}$, $i=1,\dots, d$ are mutually independent  
and that 
 their  covariance functions $R_{i}(s,t):=\mE[x_{i}(s)x_{i}(t)]$, $i=1,\dots, d$ are continuous and of finite $\rho$-variation for $\rho\geq 1$.  
  Assume  further  that one of the following two conditions  is true: 

\noindent (1) $\rho<2$; 

\noindent (2)  $\rho \geq 2$, and  the one-variable function $R_{i}(t,t), t\in[0,T]$ is of finite $\rho'$-variation for  some $\rho'\geq 1$ such that $\frac{1}{\rho}+\frac{1}{\rho'}>1$.

\end{hyp}

 
  To make sense of a Stratonovich-type integral in this setting  we will focus on integrand  of the form $\partial f:=\sum_{i=1}^{d}\frac{\partial f}{\partial x_{i}}$ for $f : (t,x)\in[0,T]\times \RR^{d}\to \RR$. 
   Let $\cp:0=t_{0}<\cdots<t_{n}=T$ be a partition of $[0,T]$.
   Consider the following compensated Riemann sum: 
   \begin{eqnarray}
\cj(\partial f, dx, \cp):=
 \sum_{k=0}^{n-1} \sum_{m=1}^{[2\rho]} \frac{1}{m!} \partial^{m}f(t_{k},x_{t_{k}})\cdot (x_{ t_{k+1}}-x_{t_{k}})^{\otimes m} ,
\notag
\end{eqnarray}
where $\partial^{m}f$ is the collection of all $m$-th order partial derivatives of $f$ in the directions of $x_{i}$'s and the product between $ \partial^{m}f(t_{k},x_{t_{k}})$ and $ (x_{ t_{k+1}}-x_{t_{k}})^{\otimes m} $   should be interpreted as a dot product. 
Denote by $|\cp|$   the mesh size of $\cp$. 
    Then  our Stratonovich-type integral is defined as follows:
      \begin{eqnarray}
\int_{0}^{T} \partial f(t,x_{t})dx_{t}: = \lim_{|\cp|\to0} \cj(\partial f, dx, \cp) .  
\label{e.stra}
\end{eqnarray}
  One can show that  such  integral  agrees with the corresponding  rough integral   whenever both integrals are well-defined (see Remark \ref{remark.rough}).

Consider   now the following Skorohod-Riemann sum: 
  \begin{eqnarray}
\cj(\partial f,\delta x, \cp) := \sum_{k=0}^{n-1} \sum_{m=1}^{[\rho]} \frac{1}{m!}\delta^{m}\Big( \partial^{m} f(t_{k},x_{t_{k}}) \be_{k}^{\otimes m}
\Big)  , 
\label{e.crss}
\end{eqnarray}
where $\be_{k}$ is the indicate function $\be_{k} = \mathbf{1}_{[t_{k},t_{k+1}]}$, $\delta^{m}$ is the collection of all $m$-th order divergence operators for the Gaussian process $x$, and $\delta^{m}$ is applied to   $\partial^{m} f(t_{k},x_{t_{k}}) \be_{k}^{\otimes m}$ componentwisely (see Definition \ref{def.skor} for more details). 
Our 
 extended Skorohod integral is then defined by:
  \begin{eqnarray}
\int_{0}^{T}\partial f(t,x_{t})\delta x_{t} := \lim_{|\cp|\to0}   \cj(\partial f,\delta x, \cp). 
\label{e.skord}
\end{eqnarray}
The Skorohod-Riemann sum in \eqref{e.crss} is motivated by a limit theorem of compensated Riemann sum in \cite{liu2024limit}. 
In \cite{hu2013stratonovich} the authors considered the following  Wick-Riemann sum:
    \begin{eqnarray}
\sum_{k=0}^{n-1} \sum_{m=1}^{N} \frac{1}{m!} \partial^{m} f(t_{k},x_{t_{k}}) \diamond \delta^{m} (\be_{k}^{\otimes m}  
 )   , 
\label{e.wrs}
\end{eqnarray}
where $N$ is equal to  $[2\rho]$ in our notation.
 It turns out that the components in our   Skorohod-Riemann sum coincides with those in the Wick-Riemann sum \eqref{e.wrs}. More precisely, we have the relation (see Proposition 4.1 (3) in  \cite{hu2013stratonovich})
 \begin{eqnarray}
\delta^{m}\Big( \partial^{m} f(t_{k},x_{t_{k}}) \be_{k}^{\otimes m}
\Big)  =  \partial^{m} f(t_{k},x_{t_{k}}) \diamond \delta^{m} (\be_{k}^{\otimes m}  
 )   
\notag
\end{eqnarray}
for a general $f$.  We will see that our result improves \cite{hu2013stratonovich} in several nontrivial ways.

In the following is  our main result (see Theorem \ref{thm.main} for a more precise statement). 
For convenience, let us denote $\partial^{2}_{ii} = \frac{\partial^{2}}{\partial x_{i}^{2}}$ and the   Riemann sum: 
\begin{eqnarray}
\cj(\partial^{2}f, dR, \cp) := \sum_{k=0}^{n-1}  \sum_{i=1}^{d}  {\partial^{2}_{ii} f}  (t_{k},x_{t_{k}})\cdot \Big(R_{i}(t_{k+1},t_{k+1}) - R_{i}(t_{k},t_{k}) \Big).   
\label{e.fdr}
\end{eqnarray}

\begin{theorem}\label{thm.intro}
Let  $f $ be a function that is   sufficiently smooth, and  $ x $ be  a  Gaussian process  satisfying Hypothesis \ref{hyp.1}. 
Then the limit in  \eqref{e.stra} exists and we have the  relation:
\begin{eqnarray}
  \int_{0}^{T}\partial f(t,x_{t})d  x_{t}  =  \lim_{|\cp|\to0} \lp \cj(\partial f,\delta x, \cp) + \frac12 \cj(\partial^{2}f, dR, \cp) \rp.
\notag
\end{eqnarray}
Moreover, if   one of the following holds:

\noindent (a)  $\rho <3/2$\,,  
 
\noindent (b) $\frac{1}{2\rho}+\frac{1}{\rho'}>1$\,, 
 
\noindent then   both  $\cj(\partial f,\delta x, \cp) $ and $  \cj(\partial^{2}f, dR, \cp)$ converge, and we have    the        conversion formula: 
 \begin{eqnarray}
  \int_{0}^{T}\partial f(t,x_{t})d  x_{t}  = \int_{0}^{T} \partial f(t,x_{t})\delta x_{t} +\frac12 \sum_{i=1}^{d}\int_{0}^{T} {\partial^{2}_{ii} f}  (t,x_{t})dR_{i}(t,t) , 
\label{e.dxdelta11}
\end{eqnarray}
where    $\int_{0}^{T} {\partial^{2}_{ii} f}  (t,x_{t})dR_{i}(t,t)$ should be interpreted as a Young integral.  
\end{theorem}

In the following we make some remarks about Theorem \ref{thm.intro}. 

\begin{remark}
It is well-known that  $x$ is of finite $p$-variation almost surely for $p>2\rho$. So the   additional condition    $\rho<3/2$ or  $\frac{1}{\rho'}+\frac{1}{2\rho}>1$   ensures  that  the Young integral  $\int_{0}^{T} {\partial^{2}_{ii} f}  (t,x_{t})dR_{i}(t,t)$ is well-defined (see also Remark \ref{remark.young}). 
\end{remark}
\begin{remark}
Note that the condition  $\frac{1}{\rho'}+\frac{1}{2\rho}>1$ is always satisfied when $\rho'=1$. In other words,   
Theorem \ref{thm.intro} shows  that    the conversion formula \eqref{e.dxdelta11}   holds   whenever the diagonoal functions $R_{i}(t,t)$, $t\in[0,T]$ are of bounded variation.  
\end{remark}
\begin{remark}
   Theorem \ref{thm.intro} implies       the  following   It\^o formula: 
\begin{eqnarray}
f(T,x_{T}) &=& f(0,0)+\int_{0}^{T}\frac{\partial f}{\partial t}(t, x_{t})dt + \int_{0}^{T} {\partial f} (t, x_{t})\delta x_{t} + \frac12\sum_{i=1}^{d} \int_{0}^{T} {\partial^{2}_{ii} f} (t, x_{t})dR_{i}(t,t) . 
\notag
\end{eqnarray}
\end{remark}

 Theorem \ref{thm.intro} improves results of \cite{hu2013stratonovich} in several ways which we highlight in the following: 
   
   (i) We do not assume that  the Gaussian process $x$ generates a geometric rough path.  
   
    (ii)  Theorem \ref{thm.intro} shows that   the $[\rho]$-th order    expansion   is sufficient to  guarantee the convergence of 
      the Skorohod-Riemann sum
    \eqref{e.crss}, and consequently, of \eqref{e.wrs} for all $N\geq [\rho]$.  
   This improves the results of \cite{hu2013stratonovich}, in which   the $[2\rho]$-th order expansion is considered. Note that   the possibility of reducing the order of expansion was suggested  in  \cite[Section 4.3]{hu2013stratonovich}. 
 
 (iii) 
 We make minimal assumptions   for the Gaussian process or its covariance function.  

 The proof of Theorem \ref{thm.intro} combines  Malliavin calculus   with    rough path techniques  such as   the super-additivity of    a 2D control, which are summarized in Section~\ref{section.prelim}.   We define the  Stratonovich-type integral by examining     the Taylor expansion of the function $f(t,x)$, thereby eliminating  the need for a geometric rough path. This definition  also accommodates   non-autonomous integrands for the conversion formula.   
 
The paper is structured as follows. In Section \ref{section.prelim} we recall   basic results of Malliavin calculus and Gaussian processes and state some useful lemmas. In Section \ref{section.1d} we prove Theorem \ref{thm.intro} when $x $ is a  one-dimensional Gaussian process.   
 Then, in Section \ref{section.md} we extend Theorem \ref{thm.intro} to the multidimensional case.

\subsection{Notation}\label{section.notation} 
Throughout the paper we work on a  probability space $(\Omega, \mathscr{F}, P)$. If $X$ is a random variable, we denote by $\| X \|_{L^{p}}  $ the $L^{p} $-norm of $X$.
The letter $K$  stands  for  a constant independent of    any important parameters which can change from line to line. We write $A\lesssim B$ if there is a constant $K>0$ such that $A\leq KB$. We denote $[a]  $   the integer part of   $a$. 
We denote $\NN=\{1,2,\cdots\}$ and $\bar{\NN} =\NN\cup\{0\}$. 



\section{Preliminary results}\label{section.prelim}

In this section we recall the elements of     Malliavin calculus and Gaussian processes  and also  prepare some useful lemmas.

\subsection{Tensor product of Hilbert spaces}  
We briefly recall the basics of tensor product of Hilbert spaces. A detailed discussion on this topic can be found in e.g.  \cite{janson1997gaussian, kadison1986fundamentals, reed1972methods}.

Let $\ch_{1}$,\dots, $\ch_{m}$ be  real separable Hilbert spaces with inner products $\langle\cdot,\cdot\rangle_{\ch_{i}}$, $i=1,\dots,m$. Denote by $\ch_{1}\otimes \cdots\otimes \ch_{m}$ the tensor product of the   vector spaces $\ch_{1}$,\dots, $\ch_{m}$. We  define the inner product $\langle \cdot, \cdot\rangle_{\ch_{1}\otimes \cdots\otimes \ch_{m}}$ on $\ch_{1}\otimes \cdots\otimes \ch_{m}$ such that for $h_{i},g_{i}\in \ch_{i}$, $i=1,\dots, m$ we have:
\begin{eqnarray}
\langle h_{1}\otimes\cdots\otimes  h_{m}, g_{1}\otimes\cdots\otimes  g_{m} \rangle_{\ch_{1}\otimes \cdots\otimes \ch_{m}}  =  \langle h_{1}, g_{1}\rangle_{\ch_{1}} \cdots \langle h_{m}, g_{m}\rangle_{\ch_{m}}   . 
\label{e.prod}
\end{eqnarray}
Then we can  obtain a Hilbert space by taking  the completion of $\ch_{1}\otimes \cdots\otimes \ch_{m}$ under the inner product \eqref{e.prod}.  This Hilbert space will still be denoted by $\ch_{1}\otimes \cdots\otimes \ch_{m}$. Furthermore, the association law holds, namely,    there is an   unitary transformation   $U$ from $\ch_{1}\otimes \cdots \otimes \ch_{m_{1}+m_{2}} $ to  
$ 
(\ch_{1}\otimes \cdots\otimes \ch_{m_{1}})\otimes (\ch_{m_{1}+1}\otimes \cdots\otimes \ch_{m_{1}+m_{2}})   
$ 
such that 
\begin{eqnarray}
U(h_{1}\otimes \cdots \otimes h_{m_{1}+m_{2}}) =  (h_{1}\otimes \cdots \otimes h_{m_{1}})  \otimes (h_{m_{1}+1}\otimes \cdots \otimes  h_{m_{1}+m_{2}}) . 
\notag
\end{eqnarray}

\subsection{Elements of Malliavin calculus}  \label{section.m}
In this subsection we  recall the elements   of Malliavin calculus and state some preliminary results. We refer the reader to e.g.  \cite{hu2016analysis, nourdin2012normal, N06} for further details.  
Let $\ch$ be a real separable Hilbert space with inner product $\langle \cdot, \cdot\rangle_{\ch}$ and norm $\|\cdot\|_{\ch}=\langle \cdot, \cdot\rangle_{\ch}^{1/2}$. For $q\in\NN$ we denote by $\ch^{\otimes q}$ and $\ch^{\odot q}$, respectively, the $q$th tensor product and the $q$th symmetric tensor product of $\ch$.  Let $X = \{X(h): h\in\ch\}$ be an isonormal Gaussian process over $\ch$, namely, $X$ is a centered Gaussian family defined on some probability space $(\Omega, \cf , P)$ such that $\mE[X(g)X(h)]=\langle g,h \rangle_{\ch}$ for every $g, h\in\ch$.  Let $\mathit{S}$   denote the set of   random variables of the form
$
F=f(W(h_{1}), \dots, W(h_{n}))$,  
where  $n\geq 1$,  $h_{1}, \dots, h_{n}\in\ch$,  and $f:\RR^{m}\to \RR$ is a $C^{\infty}$-function such that $f$ and its partial derivatives have at most polynomial growth. The $l$th Malliavin derivative of $F\in \mathit{S}$ is   a random variable with values in $\ch^{\odot p}$ defined by:
\begin{eqnarray*}
D^{l} F &=& \sum_{i_{1}, \dots, i_{l}=1}^{m} \frac{\partial^{l}f}{\partial x_{i_{1}}\cdots \partial x_{i_{l}} } (X(h_{1}), \dots, X(h_{m})) h_{i_{1}}\otimes \cdots \otimes h_{i_{l}}. 
\end{eqnarray*}
For any $p\geq 1$ and any integer $k\geq 1$, we define the Sobolev space $\DD^{k,p}$ as the closure of $\mathit{S}$ with respect to the norm:
\begin{eqnarray}
\left\| F\right\| _{k,p}^{p}=\mathbb{E}\left[ \left| F\right| ^{p}\right] +
\sum_{l=1}^{k} \mathbb{E} \big[  
 \| D^{l} F \| ^p _{\mathcal{H}^{\otimes l}}
 \big] .
\label{e.sobolev}
\end{eqnarray}
Then we can extend the domain of $D^{k}$ to the space $\DD^{k,p} $. 
We denote by $\delta^{ k} $ the adjoint of the derivative operator $D^{k}$. We say $u\in\text{Dom}(\delta^{ k})$ if there is a $\delta^{ k} (u)\in L^{2}(\Omega)$ such that for any $F\in \DD^{k,2}$ the following duality relation  holds:
\begin{eqnarray}
\mE( \langle u, D^{k}F \rangle_{\ch^{\otimes k}} ) = \mE[\delta^{ k}(u)F]. 
\label{e.ibp}
\end{eqnarray}
When $k=1$ we simply write $\delta^{ 1}=\delta  $. We will also use the convention  that $\text{Dom}( \delta^{0}) = L^{2}(\Omega)$ and $\delta^{0}(u)=u$.   
 Recall that for  $h\in\ch: |h|_{\ch}=1$ we have the relation:  
 \begin{eqnarray}
\delta^{ k}(h^{\otimes k}) = H_{k}(\delta (h)), 
\label{e.dkh}
\end{eqnarray}
where $H_{k}(x)= (-1)^{k}e^{x^{2}/2}\frac{d^{k}}{dx^{k}} e^{-x^{2}/2}$ is called the Hermite polynomial of order $k$. 
For monomials it is well-known that we have the following Hermite decomposition: 
\begin{eqnarray}
x^{i}  =  \sum_{0\leq 2q\leq i} a^{i}_{i-2q} {H_{i-2q}(x)}  \,, 
\qquad
\text{where}\quad
a^{i}_{i-2q} = \frac{i!}{2^{q}q!(i-2q)!}\,,   
\qquad 
i\in\NN . 
\label{e.xah}
\end{eqnarray}

Let us recall the following   result (see e.g. \cite[Proposition 2.5.4]{nourdin2012normal}):
\begin{lemma}\label{lemma.dfu}
Let $F\in \DD^{1,2}$ and $u\in \emph{Dom}( \delta)$ be such that the expectations $\mE[F^{2}\|u\|_{\ch}^{2}]$, $\mE[F^{2}\delta(u)^{2}]$, $\mE[ \langle DF,u \rangle_{\ch}^{2}]$ are finite. Then $Fu\in \emph{Dom}( \delta)$ and 
\begin{eqnarray}
F\delta (u)=\delta (Fu)+\langle DF,u \rangle_{\ch} . 
\label{e.dfu}
\end{eqnarray}
\end{lemma}

Let $\cv$ be another real separable Hilbert space. Let $S_{\cv}$ be the collection of all smooth $\cv$-valued random variables of the form $F = \sum_{j=1}^{n}F_{j}v_{j}$, where $F_{j}\in S$, $v_{j}\in \cv$, $n\in\NN$.  We define the Malliavin derivatives of $F\in S_{\cv}$  by $D^{k}F = \sum_{j=1}^{n}D^{k}F_{j}\otimes  v_{j}$. Similar to \eqref{e.sobolev}, we can extend the domain of $D^{k}$ to the space $\DD^{k,p}(\cv)$, which is defined as the closure of $S_{\cv}$ with respect to the norm 
\begin{eqnarray}
\|F\|_{\DD^{k,p}(\cv)}^{p} = \mE[\|F\|_{\cv}^{p}]+\sum_{l=1}^{k} \mE[\|D^{l}F\|_{\ch^{\otimes l}\otimes \cv}^{p}] . 
\notag
\end{eqnarray}

%
%
%
%


The  following result can be considered   an extension of Lemma \ref{lemma.dfu}.
\begin{lemma}\label{lemma.fdi}
Let $F\in \DD^{ i,2}$, $i\in\NN$ and $h\in \ch$, and set  
\begin{eqnarray}
 u_{j}=\langle D^{j}F, h^{\otimes j}\rangle_{\ch^{\otimes j}} \cdot h^{\otimes (i-j)} , \qquad j=0,\dots, i.
\label{e.hdfh}
\end{eqnarray}
Then $u_{j} \in \emph{Dom} (\delta^{i-j})$, and we have the   estimate: 
\begin{eqnarray}
\|\delta^{i-j}(u_{j})\|_{L^{2}} \leq K\cdot\|F\|_{i,2} \cdot \|h\|_{\ch}^{i}  
\label{e.dujb}
\end{eqnarray}
for all $j=0,\dots, i$, where $K$ is a constant depending on $i$ only. 
Furthermore,   we   have the   relation: 
\begin{eqnarray}
F\delta^{i}(h^{\otimes i}) = \sum_{j=0}^{i} {i\choose j} \delta^{i-j} (u_{j}) . 
\label{e.gfdh}
\end{eqnarray}
\end{lemma}
\begin{proof}
We     differentiate  $u_{j}$ in  \eqref{e.hdfh} and then take the $\|\cdot\|_{\ch^{\otimes  (i-j+L)}}$ norm to get
\begin{eqnarray}
\|D^{L}u_{j}\|_{\ch^{\otimes  (i-j+L)}} 
= \|  h^{\otimes (i-j)} D^{L}  \langle D^{j} F   ,  h^{\otimes j}\rangle_{\ch^{\otimes j}}\|_{\ch^{\otimes (j+L)}} 
\notag
\\
=   \|   D^{L}  \langle D^{j} F   ,  h^{\otimes j}\rangle_{\ch^{\otimes j}}\|_{\ch^{\otimes L}}  \cdot \| h \|_{\ch}^{ i-j} . 
\label{e.duh}
\end{eqnarray}
In the following we  bound $\|   D^{L}  \langle D^{j} F   ,  h^{\otimes j}\rangle_{\ch^{\otimes j}}\|_{\ch^{\otimes L}} $. Take  a vector $g\in \ch^{\otimes L}$. Then we have 
\begin{eqnarray}
\langle  D^{L} \langle D^{j } F   ,  h^{\otimes j}\rangle_{\ch^{\otimes j}} , g \rangle_{\ch^{\otimes L}} =   \langle D^{j+L} F   ,  h^{\otimes j}\otimes g\rangle_{\ch^{\otimes (j+L)}}  . 
\label{e.dfhg}
\end{eqnarray}
 Applying the triangle inequality to the right-hand side of \eqref{e.dfhg} we   obtain:
 \begin{eqnarray}
\langle  D^{L} \langle D^{j} F   ,  h^{\otimes j}\rangle_{\ch^{\otimes j}} , g \rangle_{\ch^{\otimes L}}\leq \|D^{j+L}F\|_{\ch^{\otimes (j+L)}}  
\cdot \| h^{\otimes j}\otimes g\|_{\ch^{\otimes (j+L)}}  
\notag
\\
= \|D^{j+L}F\|_{\ch^{\otimes (j+L)}} \cdot \|h\|_{\ch}^{j}  
\cdot \|   g\|_{\ch^{\otimes L}}  . 
\notag
\end{eqnarray}
Since $g\in\ch$ is arbitrary,  
this implies that 
\begin{eqnarray}
\|     \langle D^{j+L} F   ,  h^{\otimes j}\rangle_{\ch^{\otimes j}}\|_{\ch^{\otimes L}} \leq \|D^{j+L}F\|_{\ch^{\otimes (j+L)}} \cdot \|h\|_{\ch}^{j}   . 
\label{e.dfhn}
\end{eqnarray}
Substituting the estimate \eqref{e.dfhn} into  \eqref{e.duh} we obtain
\begin{eqnarray}
\|D^{L}u_{j}\|_{\ch^{\otimes  (i-j+L)}}  \leq \|D^{j+L}F\|_{\ch^{\otimes (j+L)}}\cdot \|h\|_{\ch}^{i} .  
\label{e.duj}
\end{eqnarray}
Taking the $L^{2}(\Omega)$ norm in both sides of \eqref{e.duj} we thus obtain  \begin{eqnarray}
\mE[\|D^{L}u_{j}\|_{\ch^{\otimes  (i-j+L)}}^{2}   ]^{1/2} \leq \|F\|_{i,2} \cdot \|h\|_{\ch}^{i}  
\label{e.dujn}
\end{eqnarray}
 for all $L=1,\dots, i-j$ and $ j=0,\dots, i$.
 It follows that $  u_{j}  \in \DD^{(i-j), 2}(\ch^{\otimes (i-j)})$ and we have the relation 
    \begin{eqnarray}
\|u_{j}\|_{\DD^{(i-j), 2}(\ch^{\otimes (i-j)})}\leq \|F\|_{i,2} \cdot \|h\|_{\ch}^{i}\,. 
\notag
\end{eqnarray}
  By Meyer's inequality (see e.g. \cite[Theorem 2.5.5]{nourdin2012normal}) we obtain that  $u_{j} \in \text{Dom} (\delta^{i-j})$ and the estimate \eqref{e.dujb} holds.

We turn to the proof of \eqref{e.gfdh}. 
Let $G$ be any  smooth random variable. That is, $G$ has the form $G= f(W(h_{1}), \dots, W(h_{L}))$, where $f\in C_{c}^{\infty}(\RR^{L})$ and $h_{1}, \dots, h_{L}\in \ch$, $L\geq1$. Since $F \in \DD^{i,2}$ and $G$ is smooth,   we have $F\cdot G\in \DD^{i,2}$, and by applying the duality relation  \eqref{e.ibp} we get:    
\begin{eqnarray}
\mE[GF\delta^{i}(h^{\otimes i}) ]= \mE[\langle D^{i}[GF],  h^{\otimes i}\rangle_{\ch^{\otimes i}}]. 
\label{e.gfd}
\end{eqnarray}
Furthermore, by applying the Leibniz rule (see e.g. \cite[Exercise 2.3.10]{nourdin2012normal}) to $D^{i}[GF]$ in \eqref{e.gfd} we obtain:
\begin{eqnarray}
 \mE[GF\delta^{i}(h^{\otimes i}) ]&=&\sum_{j=0}^{i}{i\choose j}  \mE\big[\langle D^{i-j} G \tilde{\otimes}D^{j} F ,  h^{\otimes i}\rangle_{\ch^{\otimes i}}\big] . 
\label{e.leibniz}
\end{eqnarray}
Applying \eqref{e.prod} to the inner product in the right-hand side of \eqref{e.leibniz}  yields:
\begin{eqnarray}
\mE[GF\delta^{i}(h^{\otimes i}) ] &=&
 \sum_{j=0}^{i}{i\choose j}  \mE\big[\langle D^{i-j} G   ,  h^{\otimes (i-j)}\rangle_{\ch^{\otimes (i-j)}} \cdot \langle D^{j} F   ,  h^{\otimes j}\rangle_{\ch^{\otimes j}}  \big].
\label{e.gfd1}
\end{eqnarray}
Recall that $u$ is given in  \eqref{e.hdfh}. So equation \eqref{e.gfd1}   can be written as:  
\begin{eqnarray}
\mE[GF\delta^{i}(h^{\otimes i}) ] &=&\sum_{j=0}^{i}{i\choose j}  \mE\big[\langle D^{i-j} G   ,  u \rangle_{\ch^{\otimes (i-j)}} \big]. 
\label{e.gfd2}
\end{eqnarray}

 

  We have shown that $u_{j} \in \text{Dom} (\delta^{i-j})$. So by    applying \eqref{e.ibp} to  \eqref{e.gfd2} we   obtain:     
\begin{eqnarray}
\mE[GF\delta^{i}(h^{\otimes i}) ] =\sum_{j=0}^{i}{i\choose j}  \mE[   G   \delta^{i-j}(  u_{j} ) ] . 
\label{e.gfd3}
\end{eqnarray}
Since smooth random variables are dense in $L^{2} (\Omega)$, relation \eqref{e.gfd3} implies that  $\eqref{e.gfdh}$ holds. 
This completes the proof. 
\end{proof}
We will   need the following technical lemma. 
\begin{lemma}\label{lemma.ddfhg}
Let $i,i'\in\NN$,  $g,h\in \ch$ and $F\in \DD^{ i+i',2}$.  
Then $\delta^{i}(F h^{\otimes i})\in \DD^{i', 2}$, and we have the relation:
\begin{eqnarray}
\langle D^{i'}\delta^{i}(F h^{\otimes i}), g^{\otimes i'}\rangle_{\ch^{\otimes i'}} = \sum_{j=0}^{i'\wedge i} {i\choose j} {i'\choose j}  j!  \langle h , g \rangle_{\ch}^{j}\cdot \delta^{i-j}\Big( \langle D^{i'-j}F ,  g^{\otimes (i'-j)} \rangle_{\ch^{\otimes (i'-j)}} h^{\otimes (i-j)} \Big). \qquad
\label{e.ddeltaf}
\end{eqnarray}
\end{lemma}
\begin{proof} 
As in Lemma \ref{lemma.fdi} 
we can show that $ \langle D^{i'-j}F ,  g^{\otimes (i'-j)} \rangle_{\ch^{\otimes (i'-j)}} h^{\otimes (i-j)} \in \text{Dom}(\delta^{i-j} )$, $j=0,\dots, i$, and therefore the right-hand side  of \eqref{e.ddeltaf} is well-defined. 
In the following, the proof is divided  into  two steps. In the first step, we consider the case when $i'=1$. Then in the second step we prove the lemma for all $i'$ and $ i$. 

When $i'=1$ relation \eqref{e.ddeltaf} becomes
\begin{eqnarray}
\langle D \delta^{i}(F h^{\otimes i}), g \rangle_{\ch } = 
  \delta^{i}( \langle D F ,  g  \rangle_{\ch} h^{\otimes i} )
+
i   \langle h , g \rangle_{\ch} \cdot \delta^{i-1}( F   h^{\otimes (i-1)} ) . 
\label{e.ddfhg1}
\end{eqnarray} 
In the following,  we show by induction that $\delta^{i}(F h^{\otimes i})\in \DD^{1,2}$ and that relation \eqref{e.ddfhg1} holds for all $i\in \NN$. 


Recall the relation \eqref{e.dfu}. So we have  $\delta (F h ) = F\delta(h)-\langle DF, h\rangle_{\ch}$. By differentiating   both sides of this relation in the direction of $g$, we obtain that
\begin{eqnarray}
\langle D \delta (F h ), g \rangle_{\ch } = 
  \delta ( \langle D F ,  g  \rangle_{\ch} h  )
+
    \langle h , g \rangle_{\ch} \cdot   F . 
\notag
\end{eqnarray}
This implies that $\delta (F h )\in \DD^{1,2}$ and \eqref{e.ddfhg1} holds when $i=1$. 

Applying relation \eqref{e.gfdh} with $i$ replaced by $i+1$ we obtain 
\begin{eqnarray}
\delta^{i+1}(F h^{\otimes (i+1)})  = F\delta^{i+1}(h^{\otimes (i+1)}) - \sum_{j=1}^{i+1} {i+1\choose j} \delta^{i+1-j} (\tilde{u}_{j}) ,
\label{e.dfh1}
\end{eqnarray}
where $\tilde{u}_{j}=\langle D^{j}F, h^{\otimes j}\rangle_{\ch^{\otimes j}} \cdot h^{\otimes (i+1-j)} $,  $ j=0,\dots, i+1$. 
Differentiating both sides of \eqref{e.dfh1} in $g$ we obtain:
\begin{eqnarray}
\langle D\delta^{i+1}(F h^{\otimes (i+1)})  , g \rangle_{\ch}
 = 
\langle DF  , g \rangle_{\ch}
\delta^{i+1}(h^{\otimes (i+1)})+(i+1)F \delta^{i}(h^{\otimes i})
\langle h  , g \rangle_{\ch}
\notag
 \\
 - \sum_{j=1}^{i+1} {i+1\choose j} \langle D\delta^{i+1-j} (\tilde{u}_{j}), g\rangle_{\ch} . 
\label{e.ddfhg}
\end{eqnarray}
Note that $i+1-j\leq i$ in \eqref{e.ddfhg}.
By applying the induction assumption \eqref{e.ddfhg1}  with $\langle D^{j}F, h^{\otimes j}\rangle_{\ch^{\otimes j}} $ and $  i+1-j  $ in places of $ F $ and $i$  we have $\delta^{i+1-j} (\tilde{u}_{j})\in \DD^{1,2}$ and 
\begin{eqnarray}
\langle D\delta^{i+1-j} (\tilde{u}_{j}), g\rangle_{\ch}  &=& 
\delta^{i+1-j}(u_{j}' )
 +
(i+1-j)   \langle h , g \rangle_{\ch} \cdot \delta^{i-j}( u_{j} ), \qquad
\label{e.ddujg}
\end{eqnarray}
where $u_{j}' = \langle D^{j}\langle D F ,  g  \rangle_{\ch}, h^{\otimes j}\rangle_{\ch^{\otimes j}} \cdot h^{\otimes (i+1-j)} 
$ and $u_{j}$ is defined as in \eqref{e.hdfh}. 
On the other hand, 
applying  relation \eqref{e.gfdh} with $i+1$ and $ \langle D F ,  g  \rangle_{\ch} $ in the places of $i$ and $F$,  we obtain that 
\begin{eqnarray}
\delta^{i+1}\big( \langle D F ,  g  \rangle_{\ch} h^{\otimes (i+1)} \big) =  
\langle D F ,  g  \rangle_{\ch} \delta^{i+1}(h^{\otimes (i+1)}) - \sum_{j=1}^{i+1} {i+1\choose j} \delta^{i+1-j} (u_{j}'). 
\label{e.ddfgh}
\end{eqnarray}
 Substituting \eqref{e.ddujg} into \eqref{e.ddfhg} and then taking the difference between \eqref{e.ddfhg} and \eqref{e.ddfgh} we obtain: 
\begin{eqnarray}
&&\langle D\delta^{i+1}(F h^{\otimes (i+1)})  , g \rangle_{\ch} -\delta^{i+1}\big( \langle D F ,  g  \rangle_{\ch} h^{\otimes (i+1)} \big)
\notag
\\
&& =
(i+1)F \delta^{i}(h^{\otimes i})
\langle h  , g \rangle_{\ch}  - \sum_{j=1}^{i} {i+1\choose j}  (i+1-j)   \langle h , g \rangle_{\ch} \cdot \delta^{i-j}\big( \langle D^{j}F, h^{\otimes j}\rangle_{\ch^{\otimes j}}   h^{\otimes (i-j)} \big)
\notag
\\
&&=(i+1)  \langle h , g \rangle_{\ch}\cdot    \Big( F \delta^{i}(h^{\otimes i})
   - \sum_{j=1}^{i} {i\choose j}     \cdot \delta^{i-j}\big( \langle D^{j}F, h^{\otimes j}\rangle_{\ch^{\otimes j}}   h^{\otimes (i-j)} \big)
   \Big). 
\label{e.ddfhgd}
\end{eqnarray}
Note that   the summation in \eqref{e.ddfgh} does not contain   the case $j=i+1$ since we have $i+1-j=0$ in \eqref{e.ddujg}. 
Now, invoking relation \eqref{e.gfdh} in the right-hand side of \eqref{e.ddfhgd} we obtain that 
\begin{eqnarray}
\langle D\delta^{i+1}(F h^{\otimes (i+1)})  , g \rangle_{\ch} -\delta^{i+1}( \langle D F ,  g  \rangle_{\ch} h^{\otimes (i+1)} ) = (i+1)  \langle h , g \rangle_{\ch} \delta^{i}(u_{0}). 
\notag
\end{eqnarray}
This gives \eqref{e.ddfhg1} when $i +1$ is in the place of $i $. This proves relation \eqref{e.ddfhg1}, and thus \eqref{e.ddeltaf}  for $i'=1$ and all $i\in\NN$.



In the following we prove by induction in the integer $i'$ that $\delta^{i}(F h^{\otimes i})\in \DD^{i', 2}$ and relation  \eqref{e.ddeltaf} holds for all $i,i'\in \NN$. 
Assume that \eqref{e.ddeltaf}  holds for some $i'\in \NN$. It suffices to show that \eqref{e.ddeltaf} still holds when $i'+1$ is in the place of $i'$. 

First, 
by applying \eqref{e.ddfhg1} with $i-j$ and  $\langle D^{i'-j}F ,  g^{\otimes (i'-j)} \rangle_{\ch}$  in places of $F$ and $i$, we have  
\begin{eqnarray}
\Big\langle D \delta^{i-j}\Big( \langle D^{i'-j}F ,  g^{\otimes (i'-j)} \rangle_{\ch} h^{\otimes (i-j)} \Big), g\Big\rangle_{\ch} = 
  \delta^{i-j}\Big(   \langle D^{i'+1-j}F ,  g^{\otimes (i'+1-j)} \rangle_{\ch}  h^{\otimes i-j} \Big)
  \notag
\\
+
(i -j)  \langle h , g \rangle_{\ch} \cdot \delta^{i-j-1}\Big(  \langle D^{i'-j}F ,  g^{\otimes (i'-j)} \rangle_{\ch}   h^{\otimes (i-j-1)} \Big) . 
\label{e.dddfg}
\end{eqnarray}
Summing both sides of  \eqref{e.dddfg} in $j$ and then applying some elementary computations it is  easy to show that we have  
\begin{eqnarray}
\sum_{j=0}^{i'\wedge i} {i\choose j} {i'\choose j}  j!  \langle h , g \rangle_{\ch}^{j}\cdot \Big\langle D \delta^{i-j}\Big( \langle D^{i'-j}F ,  g^{\otimes (i'-j)} \rangle_{\ch} h^{\otimes (i-j)} \Big), g\Big\rangle_{\ch} 
\notag
\\
= \sum_{j=0}^{(i'+1)\wedge i} {i\choose j} {i'+1\choose j}  j!  \langle h , g \rangle_{\ch}^{j}\cdot \delta^{i-j}\Big( \langle D^{i'+1-j}F ,  g^{\otimes (i'+1-j)} \rangle_{\ch} h^{\otimes (i-j)} \Big). 
\label{e.ijijddf}
\end{eqnarray}
Note that by induction assumption the left-hand side of \eqref{e.ijijddf} is equal to 
\begin{eqnarray}
\langle D\langle D^{i'}\delta^{i}(F h^{\otimes i}), g^{\otimes i'}\rangle_{\ch^{\otimes i'}} , g\rangle_{\ch} = \langle D^{i'+1}\delta^{i}(F h^{\otimes i}), g^{\otimes (i'+1)}\rangle_{\ch^{\otimes (i'+1)}}. 
\notag
\end{eqnarray}
So relation \eqref{e.ijijddf} gives    \eqref{e.ddeltaf} with $i'+1$   in the place of $i'$. It follows that  \eqref{e.ddeltaf} holds for all $i'$ and $i$. This completes the proof. 
%
\end{proof}
The following result is a consequence of Lemma \ref{lemma.ddfhg}.
\begin{lemma}\label{lemma.dfdg}
Let $i,i'\in\NN$, $g,h\in\ch$  and $F,G\in \DD^{M, 2}$ for $M=\max(i,i')$. Then the following relation holds:
\begin{eqnarray}
\mE[\delta^{i}(F h^{\otimes i}) \cdot \delta^{i'}(Gg^{\otimes i'})] &=& \sum_{j=0}^{i\wedge i'}{i\choose j}{i'\choose j} j! \langle D^{ i'-j}F, g^{\otimes (i'-j)}\rangle_{\ch^{\otimes (i'-j)}}  
\notag
\\
&&\qquad\quad\cdot  \langle D^{ i-j}G, h^{\otimes (i- j)}\rangle_{\ch^{\otimes (i- j)}}  \cdot \langle h, g\rangle_{\ch}^{j} . \quad
\label{e.dfhdgg}
\end{eqnarray}


\end{lemma}
\begin{proof}
When $F$ and $G$ are smooth random variables, 
relation \eqref{e.dfhdgg} can be shown by   applying integration by parts \eqref{e.ibp} and   relation \eqref{e.ddeltaf}, and then \eqref{e.ibp} again. By a density argument it is easy to show that \eqref{e.dfhdgg} also holds for $F,G\in\DD^{M,2}$. 
   \end{proof}

\subsection{Gaussian process and covariance function}\label{section.cov}

Let $x$ be a one-dimensional Gaussian process on $[0,T]$ and denote the covariance function $R(s,t) = \mE[x_{s}x_{t}]$, $s,t\in[0,T]$. Recall that the rectangular increment of $R$ is defined as:  
\begin{eqnarray}
R_{uv}^{st} := R(v,t)-R(u,t)-R(v,s)+R(u,s), \qquad u,v,s,t\in[0,T]: u\leq v, s\leq t.
\notag
\end{eqnarray}
  Let $I_{i}=[a_{i},b_{i}]  $, $i=1,2$ be two sub-intervals of $[0,T]$. Let $\cp_{1} = \{a_{1}=t_{0}<\cdots<t_{n_{1}}=b_{1}\}$ and $\cp_{2} = \{a_{2}=s_{0}<\cdots<s_{n_{2}}=b_{2}\}$, $n_{1},n_{2}\in\NN$   be   partitions of $I_{i}$, $i=1,2$, respectively. 
The $\rho$-variation of a two-variable function $R$ on the domain $I_{1}\times I_{2}$ is defined by: 
\begin{eqnarray}
|R|_{\rho\text{-var},I_{1}\times I_{2}}  = \sup_{\cp_{1}\times \cp_{2}} \sum_{k =0}^{n_{1}-1}\sum_{k'=0}^{n_{2}-1}\big|R_{t_{k}t_{k +1}}^{s_{k'}s_{k'+1}}\big|^{\rho},
\notag
\end{eqnarray}
where the supremum is taken over all grid partitions of $I_{1}\times I_{2}$ (see e.g. \cite{FV10}). 
Throughout the paper, we define  $\ch$ to  be the completion of the space of indicator functions with respect to the inner product $\langle \mathbf{1}_{[u,v]},  \mathbf{1}_{[s,t]}  \rangle_{\ch} =R_{uv}^{st}$ and let the Sobolev spaces $\DD^{k,p}$ be defined as in Section \ref{section.m}.

 We will need the following  elementary   lemma: 
\begin{lemma}\label{e.alk}
The following relation holds for all $s\leq t$:       
\begin{eqnarray}
R_{0s}^{st}
= - \frac12 R_{st}^{st}+\frac12 (R_{0t}^{0t}-R_{0s}^{0s}).
\label{e.alk1}
\end{eqnarray} 
\end{lemma}
\begin{proof}
Note that   
\begin{eqnarray}
 R_{0t}^{0t} = \mE[x_{t}^{2}] = \mE[(x_{t}-x_{s}+x_{s})^{2}] = R_{st}^{st}+R_{0s}^{0s}+2R_{0s}^{st}
.
\notag
\end{eqnarray} 
  The lemma  then follows.
\end{proof}

 \subsection{Definition of $p$-variation}
 Let $x $ be a continuous path on $  [0,T]$. 
Let $\cp: s=t_{0}<\cdots<t_{n}=t$ be a partition on $[s,t]$. 
 Recall that the \emph{$p$-variation of $x$ on $[s,t]$} is defined as 
 \begin{eqnarray}
 \|x\|_{p\text{-var},[s,t]}:= \Big(\sup_{\cp} \sum_{k=0}^{n-1}|x^{1}_{t_{k}t_{k+1}}|^{p}\Big)^{1/p}   , \quad p\geq1,
\notag
\end{eqnarray}
 where the supremum is taken over all partitions $\cp$ on $[s,t]$. We say that  $x$   has finite $p$-variation if $ \|x\|_{p\text{-var},[0,T]}<\infty$. 

 \begin{remark}\label{remark.young}
Let $x$ and $y$ be continuous paths on $[0,T]$.    It is well-known (see \cite{young1936inequality}) that 
if $x$ and $y$ are of finite $p$- and $q$-variations for $p,q\geq 1$, respectively, such that $\frac{1}{p}+\frac{1}{q}>1$, 
then
 the Riemann sum $\sum_{k=0}^{n-1}y_{t_{k}}(x_{t_{k+1}}-x_{t_{k}})$ converges to the Young integral $\int_{0}^{T}y_{t}dx_{t}$ as $|\cp|\to0$.  
\end{remark}

\begin{remark}\label{remark.p}
It is well-known that for a Gaussian process $x$ such that its  covariance function $R$   is of finite $\rho$-variation for $\rho\geq1$,     the sample paths of $x$ are of finite $p$-variation  almost surely for any   $p>2\rho$; see e.g. \cite[Proposition 15.22]{FV10}.  
\end{remark}

\section{Riemann-Skorohod integral for one-dimensional Gaussian process}\label{section.1d}

In this section we   introduce the Skorohod-type Riemann integral and consider its relation to the rough integral.  
Throughout the section we assume that $x$ is a one-dimensional Gaussian process satisfying Hypothesis \ref{hyp.1}.  We start by  making some remarks about the hypothesis. 
\begin{remark}\label{remark.control}
 Assume that $R$ is a  continuous  two-variable  function on $[0,T]^{2}$ with finite $\rho$-variation for $\rho\geq 1$. Denote by $I_{1}$ and $I_{2}$ two sub-intervals of $[0,T]$. We recall a result in   \cite{friz2011note}, which states that  for     $\tilde{\rho}=\rho+\ep$ and   $\ep>0$  there exists a 2D control   (that is, a two-variable function which is continuous, zero on degenerate rectangles, and    super-additive): $I_{1}\times I_{2}\mapsto \omega(I_{1}\times I_{2})\in \RR_{+}$ such that 
 \begin{eqnarray}
|R|_{\tilde{\rho}\text{-var}, I_{1}\times I_{2}}^{\tilde{\rho} } \leq\omega(I_{1}\times I_{2}). 
\label{e.control2}
\end{eqnarray}

Given that Hypothesis \ref{hyp.1} holds, it is easy to see that for $\ep$ sufficiently small,  Hypothesis \ref{hyp.1} still holds if $\rho $ is replaced by $\tilde{\rho}=\rho+\ep$.  For the sake of conciseness, we will abuse   notation and denote this  $\tilde{\rho}$ as $\rho$ in the remainder of the paper. In particular, we will  write \eqref{e.control2} as:
\begin{eqnarray}
|R|_{ {\rho}\text{-var}, I_{1}\times I_{2}}^{ {\rho} } \leq\omega(I_{1}\times I_{2}). 
\label{e.control3}
\end{eqnarray}

\end{remark}
\begin{remark}\label{remark.rho1}
Since $R$ is of finite $\rho$-variation, the function $R(t,t)$, $t\in[0,T]$ is also of finite $\rho$-variation. This shows that condition (1) in Hypothesis \ref{hyp.1}  implies that condition (2) holds true with $\rho ' =\rho$. 
 
 To see that $R(t,t)$, $t\in[0,T]$ has finite $\rho$-variation,   we can write 
\begin{eqnarray}
R(t,t)-R(s,s) = R_{st}^{0t}+R_{0s}^{st} . 
\notag
\end{eqnarray}
Let $\cp: 0=t_{0}<\cdots<t_{n}=T$ be a partition of $[0,T]$. Then  we have   
 \begin{eqnarray}
\sum_{k=0}^{n-1} |R(t_{k+1},t_{k+1})-R(t_{k},t_{k}) |^{\rho}\leq 2^{\rho-1}\sum_{k=0}^{n-1} \big(|R_{t_{k}t_{k+1}}^{0t_{k+1}} |^{\rho}+|R_{0t_{k}}^{t_{k}t_{k+1}}|^{\rho}\big)
\\
\leq 2^{\rho-1}\sum_{k=0}^{n-1} \big(\omega([t_{k},t_{k+1}]\times [0,t_{k+1}])+ \omega( [0,t_{k}]\times [t_{k},t_{k+1}])\big),\notag  
\end{eqnarray}
where we have used relation \eqref{e.control3}. 
Now by super-additivity of $\omega$  we have
\begin{eqnarray}
\sum_{k=0}^{n-1} |R(t_{k+1},t_{k+1})-R(t_{k},t_{k}) |^{\rho}\leq 2^{\rho-1}   \omega([0,T]^{2}). 
\notag
\end{eqnarray}
Taking supremum over all partitions  $\cp$ of $[0,T]$ we obtain that the $\rho$-variation of $R(t,t)$, $t\in[0,T]$ is bounded by $ 2^{\rho-1}   \omega([0,T]^{2})$.   
\end{remark}
\begin{remark}\label{remark.rho2}
It is easy to see that under condition (2) in Hypothesis \ref{hyp.1} we can always take  $\rho'$ such that  $\rho'<2$. We will make such choice of $\rho'$ in the following discussions.  
\end{remark}

Denote   $x^{1}_{st}=x_{t}-x_{s} $ and $x^{i}_{st} =(  x^{1}_{st})^{i}/i!$, $i\in\NN$, and   $\ell=[2\rho]$. Let $\cp : 0=t_{0}<\cdots<t_{n}=T$ be a partition on $[0,T]$. Consider the compensated Riemann-Stieltjes sum: 
\begin{eqnarray}
\cj (f, dx, \cp ): =
\sum_{k=0}^{n-1} \sum_{i=1}^{\ell} f^{(i-1)}(t_{k}, x_{t_{k}})x^{i}_{t_{k}t_{k+1}} ,  
\label{e.rs}
\end{eqnarray}
where we    denote   $f^{(i)}=\frac{\partial^{i}f}{\partial x^{i}} $, $i\in\NN$.  
According to Lemma \ref{lemma.strat} (with $d=1$ and $\partial f $ replaced by $f$) and   Remark \ref{remark.p}, if   $f $ has bounded derivatives in $x$ up to order $\ell$ and continuous first-order derivative in $t$, then       we have the   convergence:  
\begin{eqnarray}
\cj (f, dx, \cp) \to \int_{0}^{T} f(t,x_{t})dx_{t}  
\label{e.strati}
\end{eqnarray}
  almost surely   as $|\cp|\to0$.
  
 Following is our definition of the   Skorohod-type Riemann integral. Recall that we  denote   $\be_{k} = \mathbf{1}_{[t_{k}, t_{k+1}]}$ for $k\in\NN$.  

 \begin{Def}\label{def.delta}
Let $f: (t,x)\mapsto f(t,x)$ be a continuous function from $  [0,T]\times \RR $ to $\RR$. 
Let $\cp : 0=t_{0}<\cdots<t_{n}=T$ be a partition on $[0,T]$. Consider  the following Skorohod-type Riemann sum:
 \begin{eqnarray}
\cj  (f, \delta x, \cp):=  \sum_{k=0}^{n-1}  \sum_{i=1}^{[\rho ] }
\frac{1}{   {i!}}
\delta^{ i}\big( f^{(i-1)}
(t_{k},x_{t_{k}})  
\be_{k}^{\otimes i }
\big). 
\label{e.skorr}
\end{eqnarray}
We define  the Skorohod-type integral: 
 \begin{eqnarray}
\int_{0}^{T} f(t,x_{t})\delta x_{t} :=\lim_{|\cp|\to0}\cj  (f, \delta x, \cp), 
\label{e.skor}
\end{eqnarray}
provided that the right-hand side of  \eqref{e.skorr} is well-defined and that the limit in \eqref{e.skor}   exists in probability. 
\end{Def} 
 
 We also define the following  Riemann sum:  
\begin{eqnarray}
 \cj  (f^{(1)}, dR, \cp)  = \sum_{k=0}^{n-1} f^{(1)}(t_{k}, x_{t_{k}}) \big( R(t_{k+1}, t_{k+1}) - R(t_{k}, t_{k})  \big) . 
\label{e.fdr1}
\end{eqnarray}
 Recall that the Gaussian process $x$ is of finite $p$-variation for $p>2\rho$. For \eqref{e.fdr1} to be convergent it suffices to assume that the diagonal $R(t,t)$, $t\in[0,T]$ is of finite $\rho'$-variation for $\frac{1}{\rho'}+\frac{1}{p}>1$, or equivalently, $\frac{1}{\rho'}+\frac{1}{2\rho}>1$.   
 

Following is   our main result in this section:
 
\begin{thm}\label{thm.dxdelta}
Let $f : (t,x)\in[0,T]\times \RR\to \RR$ be a continuous function with  bounded derivatives in $x$ up to order $2\ell-1$ and continuous first-order derivative in $t$.  
Let $x$ be a one-dimensional    Gaussian process such  that       Hypothesis \ref{hyp.1} holds.   
 Then the limit \eqref{e.strati} exists and   we have the following representation:
 \begin{eqnarray}
\int_{0}^{T} f(t,x_{t})dx_{t}   =\lim_{|\cp|\to0}  \lp \cj  (f, \delta x, \cp) + \frac12 \cj  (f^{(1)}, dR, \cp)  \rp. 
\label{e.conv}
\end{eqnarray}
Assume further that  $f^{(1)}=\frac{\partial f}{\partial x} $ is   Lipschitz in $(t,x)$, and that one of the following holds:

\noindent (1) $\rho <3/2$;

\noindent (2)       $\frac{1}{2\rho}+\frac{1}{\rho'}>1$.  

\noindent Then   both  \eqref{e.skor} and  \eqref{e.fdr1} converge as $|\cp|\to0$, and we have the      Stratonovich-Skorohod integral formula: 
 \begin{eqnarray}
 \int_{0}^{T} f(t,x_{t})d  x_{t}  = \int_{0}^{T} f(t,x_{t})\delta x_{t} +\frac12 \int_{0}^{T}\frac{\partial f}{\partial x} (t,x_{t})dR(t,t). 
\label{e.dxdelta}
\end{eqnarray}

\end{thm}

 \begin{proof}
 The proof is divided into several steps. We start by   deriving some useful algebraic results about \eqref{e.rs} (from Step 1-3).  

%
%

\noindent\emph{Step 1: Chaos decomposition of the compensated Riemann sum.}
 Denote $\si_{k}=\mE[|x^{1}_{t_{k}t_{k+1}}|^{2}]^{1/2}$ and  $X_{k} =   x^{1}_{t_{k}t_{k+1}}/\si_{k} = \delta( \be_{k} /\si_{k})$, $k=0,\dots,n-1$. It is clear that $X_{k}$ is a standard normal random variable.  
 Recall that $x^{i}_{t_{k}t_{k+1}} =(  x^{1}_{t_{k}t_{k+1}})^{i}/i!$. So we have: 
\begin{eqnarray}
x^{i}_{t_{k}t_{k+1}}  = \frac{1}{{i!}}X_{k}^{i}\cdot \si_{k}^{i}. 
\label{e.xk}
\end{eqnarray}
 By applying  \eqref{e.xah} and then \eqref{e.dkh}  to \eqref{e.xk} we obtain:  
\begin{eqnarray}
x^{i}_{t_{k}t_{k+1}} 
= \frac{1}{{i!}}\sum_{0\leq 2q\leq i} a^{i}_{i-2q} \delta^{i-2q}( \be_{k}^{\otimes (i-2q)} )  \si_{k}^{-(i-2q)}\cdot \si_{k}^{i}, 
\label{e.xhermite}
\end{eqnarray}
where $ a^{i}_{i-2q}$ is given in \eqref{e.xah}.
Substituting \eqref{e.xhermite} into \eqref{e.rs}     we get: 
%
%
%
\begin{eqnarray}
\cj(f,dx, \cp) = \sum_{k=0}^{n-1} \sum_{i=1}^{\ell}\sum_{0\leq 2q\leq i}   \frac{1}{{i!}} a^{i}_{i-2q}    f^{(i-1)}(t_{k}, x_{t_{k}}) 
 \delta^{i-2q}( \be_{k}^{\otimes (i-2q)} ) \cdot \si_{k}^{2q} . 
\label{e.xhq}
\end{eqnarray}



Apply  Lemma \ref{lemma.fdi} with $F = f^{(i-1)}(t_{k}, x_{t_{k}}) $, $ h = \be_{k}$ and with $i$ replaced by $i-2q$. Then relation \eqref{e.hdfh} becomes:  
\begin{eqnarray}
\langle D^{j}F, h^{\otimes j}\rangle_{\ch^{\otimes j}} \cdot h^{\otimes (i-2q-j)}   = f^{(i+j-1)}
(t_{k}, x_{t_{k}})  
\be_{k}^{\otimes (i-2q-j)} 
\cdot \al_{k}^{j}\, , 
\notag
\end{eqnarray}
where  we denote $\al_{k} = \langle \mathbf{1}_{[0,t_{k}]}, \be_{k} \rangle_{\ch}$.
 From    \eqref{e.gfdh} we also have: 
\begin{eqnarray}
f^{(i-1)}(t_{k}, x_{t_{k}}) 
 \delta^{i-2q}( \be_{k}^{\otimes (i-2q)} )   =
  \sum_{j=0}^{i-2q}  
\binom{i-2q}{j}
\delta^{ i-2q-j}\big( f^{(i+j-1)}
(t_{k}, x_{t_{k}})  
\be_{k}^{\otimes (i-2q-j)}
\big) 
\cdot \al_{k}^{j}\,.
\label{e.xhq2}
\end{eqnarray} 
Reporting \eqref{e.xhq2} to \eqref{e.xhq}  we   obtain:  
\begin{eqnarray}
\cj (f, dx, \cp ) =    \sum_{i=1}^{\ell}        \sum_{0\leq 2q\leq i} \sum_{j=0}^{i-2q }  
 \cn({i,q,j})\, ,   
\label{e.cjd}
\end{eqnarray}
where
\begin{eqnarray}
\cn({i,q,j}) &:=&  \sum_{k=0}^{n-1}  
 \si_{k}^{ 2q} \frac{a^{i}_{i-2q}}{   {i!}}\binom{i-2q}{j}
  \cdot\delta^{ i-2q-j}\big( f^{(i+j-1)}
(t_{k}, x_{t_{k}})  
\be_{k}^{\otimes (i-2q-j)}
\big)
 \cdot \al_{k}^{j}\, . 
\label{e.cm}
\end{eqnarray}

\noindent\emph{Step 2: Some index sets.} 
We introduce  the following useful  index sets:  
\begin{eqnarray}
\ca &=& \{(i,q,j)\in\bar{\NN}^{3}: 1\leq i\leq \ell,\,   2q \leq i,   \,   j \leq i-2q\}, 
\label{e.aset}
\\
\ca_{L } &=& \{(i,q,j)\in \ca:   i-2q-j=L \}, 
\qquad \qquad\qquad  L=0,\dots, \ell, 
\notag
\\
\ca_{L\tau} &=& \{(i,q,j)\in \ca_{L}: q+j=\tau \} , 
\qquad \qquad \qquad\qquad \tau=0, \dots,    \ell-L,
\label{e.calt1}
\\
\ci &=& \{(L, \tau)\in \bar{\NN}^{2}:   L\leq \ell,   \tau\leq \ell-L\} .
\label{e.ci}  
\end{eqnarray}
It is clear that we have the relation: 
\begin{eqnarray}
\ca= \bigcup_{L=0}^{\ell}\ca_{L}= \bigcup_{L=0}^{\ell}\bigcup_{\tau=0}^{\ell-L} \ca_{L\tau}= \bigcup_{(L, \tau)\in \ci} \ca_{L\tau}\,. 
\label{e.caca}
\end{eqnarray}
  From  \eqref{e.cjd}   we can thus write   
\begin{eqnarray}
\cj (f, dx, n ) 
=\sum_{(i,q,j)\in \ca }  \cn({i,q,j})=
\sum_{(L, \tau)\in \ci} \sum_{(i,q,j)\in\ca_{L\tau}}\cn({i,q,j})\,.   
 \label{e.cjd1}
\end{eqnarray}

Since $i-2q-j=L$ and $q+j=\tau$, we have \begin{eqnarray}
i+j=L+2\tau. 
\label{e.ijL}
\end{eqnarray}
 On the other hand, recall   that $ a^{i}_{i-2q}$ is defined  in~\eqref{e.xah}. This gives:  
 \begin{eqnarray}
\frac{a^{i}_{i-2q}}{   {i!}}\binom{i-2q}{j} = \frac{1}{2^{q}q!j!L!}. 
\label{e.aij}
\end{eqnarray}
 Substituting   relations \eqref{e.ijL}-\eqref{e.aij} into \eqref{e.cjd1} we obtain:   
\begin{eqnarray}
\cj (f, dx, \cp )   =
\sum_{(L, \tau)\in \ci} \sum_{(i,q,j)\in\ca_{L\tau}}\cm({L,\tau,i,q,j}) \,, 
\label{e.jfdxn}
\end{eqnarray}

where    
\begin{eqnarray}
\cm({L,\tau,i,q,j})
&=&   
    \sum_{k=0}^{n-1}  
\si_{k}^{ 2q} \frac{1}{2^{q}q!j!L!}
   \cdot\delta^{ L}\big( f^{(L+2\tau-1)}
(t_{k}, x_{t_{k}})  
\be_{k}^{\otimes L}
\big)\cdot \al_{k}^{j} \, .
\label{e.cmalt}
\end{eqnarray}

In the following we   make some observations about the  index sets. 
First,   by the definition of $\ca_{L}$ it is   easy to see that  we have
\begin{eqnarray}
\ca_{L} = \{(i,q,j)\in\bar{\NN}^{3}: 1\leq i\leq \ell ,\,    i-2q-j=L\}  , 
\qquad L=0,\dots, \ell. 
\label{e.cal}
\end{eqnarray}
From  \eqref{e.cal} one can show that for   $\tau: 0\leq 2\tau \leq   \ell-L$ such that $(L,\tau)\neq (0,0)$   we have
\begin{eqnarray}
\ca_{L\tau} = \{ (L+\tau,0,\tau), (L+\tau+1, 1, \tau-1),\dots, (L+2\tau, \tau, 0) \}. 
\label{e.calt}
\end{eqnarray}
Indeed, it is easy to verify that the elements in the right-hand side of \eqref{e.calt} belongs to~\eqref{e.calt1}. Note that the condition $0\leq 2\tau \leq   \ell-L$ is required to make sure that the elements satisfy the condition $i\leq \ell$ in \eqref{e.cal}. On the other hand, using relation   \eqref{e.ijL} and $i-2q-j=L$ it is clear that \eqref{e.calt1} is contained in the set \eqref{e.calt}. 
  
In the special case when $\tau=0$ and    $L=1,\dots, \ell$, relations \eqref{e.cal}-\eqref{e.calt} show that   $\ca_{L0}=\{(L, 0, 0)\}$  and  $\ca_{\ell}=\ca_{\ell0}=\{(\ell, 0,0)\}$. Note also  that since $i\geq 1$   we have $\ca_{00}=\emptyset$.  



%

\noindent\emph{Step 3: A further decomposition.}  
Recall that $\al_{k} = \langle \mathbf{1}_{[0,t_{k}]}, \be_{k} \rangle_{\ch}$. By Lemma \ref{e.alk} we have $\al_{k} = -\frac12 \si_{k}^{2}+\frac12 (R_{0t_{k+1}}^{0t_{k+1}} - R_{0t_{k}}^{0t_{k}} )$.  
Applying the binomial expansion 
 to $\al_{k}^{j}$ in  
\eqref{e.cmalt} gives:
\begin{eqnarray}
 \cm(L, \tau, i,q,j) =  \sum_{j'=0}^{j}  \cm(L, \tau, i,q,j, j') , 
\label{e.cmjp}
\end{eqnarray}
where  
\begin{eqnarray}
\cm(L, \tau, i,q,j, j') &=& \sum_{k=0}^{n-1} \delta^{ L}\big( f^{(L+2\tau-1)}
(t_{k}, x_{t_{k}})  
\be_{k}^{\otimes L}
\big)   
 \si_{k}^{ 2(q+j')}  
 \notag
\\
&& \qquad\cdot \frac{1}{2^{q}q! j!L!}
\cdot   {j\choose j'} (-\frac12 )^{j'}    \cdot (\frac12)^{j-j'} (R_{0t_{k+1}}^{0t_{k+1}} - R_{0t_{k}}^{0t_{k}} )^{j-j'}   . 
\label{e.cmcajp}
\end{eqnarray}

In the following steps we consider $\cm(L, \tau, i,q,j, j') $ in six cases:
 
\noindent(i)\quad $ j-j'\geq2$;

\noindent(ii)\quad $ j-j'=1 $, $L+\tau >1$;

\noindent(iii)\quad  $ j=j'$, $L+\tau>\rho$;

\noindent(iv)\quad $j-j'=1$, $L+\tau=1$;

\noindent(v)\quad $j=j'$, $L+\tau \leq \rho$, $\tau >0$;

\noindent(vi)\quad $j=j'$, $L  \leq \rho$, $\tau =0$. 

\noindent
Note that 
it is clear that the above six cases are mutually exclusive, and that their union covers all possible values of $(L, \tau, i,q,j, j') $. So from \eqref{e.jfdxn} and \eqref{e.cmjp} we have
\begin{eqnarray}
\cj (f, dx, \cp )   = \lp \sum_\text{case (i)}+\cdots + \sum_\text{case (vi)} \rp \cm(L, \tau, i,q,j, j') . 
\label{e.caseall}
\end{eqnarray}


\noindent\emph{Step 4: Case (i) and (ii) - a preliminary estimate.}  
By   computing the second moment of  \eqref{e.cmcajp} we get the  following estimate: 
\begin{eqnarray}
\mE[|\cm(L, \tau, i,q,j, j')|^{2}]&\lesssim& 
 \sum_{k,k'=0}^{n-1}\Big|\mE\Big[ \delta^{ L}\big( f^{(L+2\tau-1)}
(t_{k}, x_{t_{k}})  
\be_{k}^{\otimes L}
\big)\cdot \delta^{ L}\big( f^{(L+2\tau-1)}
(t_{k'}, x_{t_{k'}})  
\be_{k'}^{\otimes L}
\big) \Big]\Big|   
\notag
\\
&&\quad\cdot  \si_{k}^{ 2(q+j')}   \si_{k'}^{ 2(q+j')}   \cdot
\Big|R_{0t_{k+1}}^{0t_{k+1}} - R_{0t_{k}}^{0t_{k}} \Big|^{j-j'}\cdot
\Big|R_{0t_{k'+1}}^{0t_{k'+1}} - R_{0t_{k'}}^{0t_{k'}} \Big|^{j-j'}.
\label{e.eml2}
\end{eqnarray}
Now by applying Lemma \ref{lemma.dfdg} to the expected value in the right-hand side of \eqref{e.eml2}  we obtain: 
\begin{eqnarray}
\mE\big[\big|\cm(L, \tau, i,q,j, j')\big|^{2}\big] &\leq&\sum_{a=0}^{L} \sum_{k,k'=0}^{n-1}\Big|\langle \be_{k}, \be_{k'} \rangle_{\ch}^{a}\cdot
 \langle \be_{k'},\mathbf{1}_{[0,t_{k}]} \rangle_{\ch}^{L-a}\cdot
  \langle \be_{k}, \mathbf{1}_{[0,t_{k'}]}  \rangle_{\ch}^{L-a}\Big|   
\notag
\\
&&\cdot  \si_{k}^{ 2(q+j')}   \si_{k'}^{ 2(q+j')}   \cdot
\Big|R_{0t_{k+1}}^{0t_{k+1}} - R_{0t_{k}}^{0t_{k}} \Big|^{j-j'}\cdot
\Big|R_{0t_{k'+1}}^{0t_{k'+1}} - R_{0t_{k'}}^{0t_{k'}} \Big|^{j-j'} . 
\qquad
\label{e.eml2i}
\end{eqnarray}
Let $p_{i}$, $i=1,2,3$ be constants such that 
\begin{eqnarray}
 p_{i}>1,\quad i=1,2,3
 \qquad
 \text{and}
 \qquad
 \frac{1}{p_{1}}+\frac{1}{p_{2}}+\frac{1}{p_{3}}=1.
\label{e.pi}
\end{eqnarray}
 By applying   H\"older's inequality to the right-hand side of \eqref{e.eml2i} we obtain:
\begin{eqnarray}
 \mE[|\cm(L, \tau, i,q,j, j')|^{2}]  
 \leq  \sum_{a=0}^{L}  J_{1}\cdot  J_{2}\cdot  J_{3}\,, 
\label{e.cml2bd1} 
\end{eqnarray}
where
\begin{eqnarray}
J_{1} &=& \Big( \sum_{k,k'=0}^{n-1}\big|\langle \be_{k}, \be_{k'} \rangle_{\ch}\big|^{ap_{1}}\Big)^{1/p_{1}}
\label{e.j1j2}
\\
J_{2} &=&  \Big(\sum_{k,k'=0}^{n-1} 
 \big|\langle \be_{k'},\mathbf{1}_{[0,t_{k}]} \rangle_{\ch} \cdot
  \langle \be_{k}, \mathbf{1}_{[0,t_{k'}]}  \rangle_{\ch}\big|^{(L-a)p_{2}}   \cdot  (\si_{k}    \si_{k'})^{ 2(q+j')p_{2}} \Big)^{1/p_{2}}
\label{e.j1j21}
\\
J_{3} &=& \Big(\sum_{k,k'=0}^{n-1}  \big|R_{0t_{k+1}}^{0t_{k+1}} - R_{0t_{k}}^{0t_{k}} \big|^{(j-j')p_{3}}\cdot 
\big|R_{0t_{k'+1}}^{0t_{k'+1}} - R_{0t_{k'}}^{0t_{k'}} \big|^{(j-j')p_{3}} \Big)^{1/p_{3}}. 
\notag
\end{eqnarray}


 In the following we specify our choices of $p_{i}$'s. Recall that by Hypothesis \ref{hyp.1} together with Remark \ref{remark.rho1}-\ref{remark.rho2} we can always find $\rho'<2$ such that $\frac{1}{\rho}+\frac{1}{\rho'}>1$.  
  We first show that in both cases (i) and   (ii), 
we have the following relation:  
\begin{eqnarray}
\frac{a}{\rho} +\frac{L-a+q+j'}{\rho} + \frac{ j-j'}{\rho'} >1    . 
\label{e.ineq}
\end{eqnarray}
We rewrite \eqref{e.ineq} as:
\begin{eqnarray}
 {(L+q+j)}  + ( j-j')(\frac{\rho}{\rho'}-1) >\rho . 
\label{e.ineq1}
\end{eqnarray}

In case (i), since $j-j'\geq 2$ we have $L+q+j\geq 2$ and thus
\begin{eqnarray}
{(L+q+j)}  + ( j-j')(\frac{\rho}{\rho'}-1)  \geq 2+2(\frac{\rho}{\rho'}-1)  =2\frac{\rho}{\rho'}. 
\notag
\end{eqnarray}
Since $\rho'<2$ (see Remark \ref{remark.rho2}) relation \eqref{e.ineq1} and thus \eqref{e.ineq} follows. 

We turn to case (ii). In this case we have $L+q+j=L+\tau\geq 2$ and $j-j'=1$. It follows that  
\begin{eqnarray}
{(L+q+j)}  + ( j-j')(\frac{\rho}{\rho'}-1)  \geq 2+ \frac{\rho}{\rho'}-1=1+\frac{\rho}{\rho'}. 
\notag
\end{eqnarray}
Since $\frac{1}{\rho}+\frac{1}{\rho'}>1$, we obtain that the above is strictly greater than $\rho$. We again have relation \eqref{e.ineq1} and thus \eqref{e.ineq}. In summary, equation  \eqref{e.ineq} holds for both case (i) and (ii).  

 Relation \eqref{e.ineq} implies that we can find $p_{i}$, $i=1,2,3$ such that relations in \eqref{e.pi} holds, and we have  
\begin{eqnarray}
\frac{a}{\rho}>\frac{1}{p_{1}} , 
\qquad
\frac{L-a+q+j'}{\rho} >\frac{1}{p_{2}} ,
\qquad \frac{ j-j'}{\rho'} >\frac{1}{p_{3}} , 
\notag
\end{eqnarray}
or equivalently, 
\begin{eqnarray}
ap_{1}>\rho , 
\qquad
(L-a+q+j')p_{2} >\rho, 
\qquad
(j-j')p_{3}>\rho'. 
\label{e.p2}
\end{eqnarray}

\noindent\emph{Step 5: Case (i) and (ii) - convergence of $\cm (L, \tau, i,q,j, j')$.} 
Since  $ap_{1}>\rho$  we have:
\begin{eqnarray}
J_{1}\leq \Big(\max_{k,k'}\big|\langle \be_{k}, \be_{k'} \rangle_{\ch}\big|^{ap_{1}-\rho} \cdot \sum_{k,k'=0}^{n-1} |\langle \be_{k}, \be_{k'} \rangle_{\ch}|^{\rho}\Big)^{1/p_{1}} . 
\label{e.j1bd}
\end{eqnarray}
Note that $\langle \be_{k}, \be_{k'} \rangle_{\ch}=R_{t_{k}t_{k+1}}^{t_{k'}t_{k'+1}}$. Substituting this into \eqref{e.j1bd}, and applying relation \eqref{e.control3} in Remark \ref{remark.control} and then  the additivity of the 2D control $\omega$ we   obtain:
\begin{eqnarray}
J_{1}\leq \Big(\max_{k,k'}\big|R_{t_{k}t_{k+1}}^{t_{k'}t_{k'+1}}\big|^{ap_{1}-\rho} \cdot  \omega([0,T]^{2})\Big)^{1/p_{1}} . 
\label{e.j1c}
\end{eqnarray}
 By uniform continuity of $R$ we thus have $J_{1}\to0$ as $|\cp|\to0$.  

We turn to the estimate of $J_{2}$. By definition of $\rho$-variation of $R$ together with relation \eqref{e.control3} we have
\begin{eqnarray}
|\langle \be_{k},\mathbf{1}_{[0,t_{k'}]} \rangle_{\ch}| \leq \omega([t_{k}, t_{k+1}]\times[0,T])^{1/\rho}, 
\notag\\
\si_{k}^{2}= |\langle \be_{k},\be_{k} \rangle_{\ch}| \leq 
\omega([t_{k}, t_{k+1}]\times[0,T])^{1/\rho}
 .  
\label{e.be1bd}
\end{eqnarray}
 Applying \eqref{e.be1bd}   to $J_{2}$ we obtain:
 \begin{eqnarray}
J_{2} \leq \Big(\sum_{k,k'=0}^{n-1}  \omega([t_{k}, t_{k+1}]\times[0,T])^{(L-a+q+j')p_{2}/\rho}\cdot  \omega([t_{k'}, t_{k'+1}]\times[0,T])^{(L-a+q+j')p_{2}/\rho}
\Big)^{1/p_{2}}
\notag
\\
=\Big(\sum_{k=0}^{n-1}   \omega([t_{k}, t_{k+1}]\times[0,T])^{(L-a+q+j')p_{2}/\rho} 
\Big)^{2/p_{2}}. 
\notag
\end{eqnarray}
Recall that from \eqref{e.p2} we have $(L-a+q+j')p_{2}>\rho$. As in \eqref{e.j1c}    we can thus show that $J_{2}\to 0$ as $|\cp|\to0$. 
A   similar argument together with the fact that $R_{0t}^{0t}$, $t\in [0,T]$ is of finite $\rho'$-variation (see Remark \ref{remark.rho1}) shows that we also have $J_{3}\to0$ as $|\cp|\to0$. 

Recall the estimate \eqref{e.cml2bd1}. So the convergences of $J_{i}$, $i=1,2,3$ imply that \begin{eqnarray}
\mE[|\cm(L, \tau, i,q,j, j')|^{2}] \to 0   
\label{e.cmconv}
\end{eqnarray}
as $|\cp|\to0$
 in both cases (i) and (ii).

\noindent\emph{Step 6: Case (iii).} 
As for cases (i) and (ii), it can be shown that  
 there exist constants $p_{1}, p_{2}>1$  such that  
\begin{eqnarray}
\frac{1}{p_{1}}+\frac{1}{p_{2}} =1 , 
\qquad
ap_{1}>\rho , 
\qquad
(L-a+q+j')p_{2} >\rho. 
\label{e.p1p2}
\end{eqnarray}
Note that since $j=j'$ we have $(R_{0t_{k+1}}^{0t_{k+1}} - R_{0t_{k}}^{0t_{k}})^{j-j'}=1 $ in  
\eqref{e.cmcajp}. 
Applying H\"older inequality to \eqref{e.eml2i}   we have 
\begin{eqnarray}
\mE[|\cm(L, \tau, i,q,j, j')|^{2}]  
 \leq  \sum_{a=0}^{L}  J_{1}\cdot  J_{2}\,, 
\notag
\end{eqnarray}
where $J_{1}$ and $J_{2}$ are given as in \eqref{e.j1j2}-\eqref{e.j1j21} with $p_{1}$ and $p_{2}$ defined in \eqref{e.p1p2}. A similar argument as in the previous step then shows that the convergence \eqref{e.cmconv} also holds in case~(iii).

In summary of the convergence of $\cm(L, \tau, i,q,j, j')$ for case (i)-(iii) obtained in Step 5-6, we have shown that
\begin{eqnarray}
 \lim_{|\cp|\to0}\lp \sum_\text{case (i)}+  \sum_\text{case (ii)}+  \sum_\text{case (iii)} \rp \cm(L, \tau, i,q,j, j')=0.
\label{e.caseiii}
\end{eqnarray}





\noindent\emph{Step 7: Case (iv).} 
In this case   we   have $j=j'+1\geq 1$. This 
implies that $j=1$ and thus $L=q=j'=0$ and $i=\tau=1$, and so 
\begin{eqnarray}
\sum_\text{case (iv)}\cm(L, \tau, i,q,j, j') = \cm(0, 1, 1,0,1, 0) = \sum_{k=0}^{n-1}f^{(1)}(t_{k}, x_{t_{k}}) \cdot \frac12 (R_{0t_{k+1}}^{0t_{k+1}} - R_{0t_{k}}^{0t_{k}} ) 
\notag
\\
 = \frac12 \cj  (f^{(1)}, dR, \cp)    . 
\label{e.youngf1}
\end{eqnarray}

\noindent\emph{Step 8: Case (v).} In this case
  \eqref{e.cmcajp}   becomes 
\begin{eqnarray}
 \cm(L, \tau, i,q,j, j) =      \sum_{k=0}^{n-1}\delta^{ L}\big( f^{(L+2\tau-1)}
(t_{k}, x_{t_{k}})  
\be_{k}^{\otimes L}
\big)\cdot  \si_{k}^{2\tau}
 \cdot  \frac{1}{2^{q}q!j!L!}\cdot(-\frac12 )^{j} 
 ,
\label{e.cmjj}
\end{eqnarray}
and  we have $L+\tau\leq[\rho]$. So  $L+2\tau\leq 2L+2\tau \leq 2[\rho]$. Recall also that    $\ell= [2\rho]$. It follows that we have $L+2\tau\leq \ell$.
This implies that     the relation~\eqref{e.calt} holds. Summing up \eqref{e.cmjj} for $(i,q,j)\in \ca_{L\tau}$ and taking into account \eqref{e.calt} we thus have:  
\begin{eqnarray}
\sum_{(i,q,j)\in \ca_{L\tau}} \cm(L, \tau, i,q,j,j ) =\sum_{k=0}^{n-1}\delta^{ L}\big( f^{(L+2\tau-1)}
(t_{k}, x_{t_{k}})  
\be_{k}^{\otimes L}
\big)\cdot \si_{k}^{ 2\tau }\frac{1}{L!}   
\notag
\\
\cdot 
\sum_{(q,j):0\leq q+j\leq\tau} \frac{1}{2^{q}} \frac{1}{q!j!}\cdot(-\frac12 )^{j} 
 . 
\label{e.cmcalt}
\end{eqnarray}
Note that the second summation in \eqref{e.cmcalt} is equal to $(\frac12 -\frac12)^{\tau} \cdot \frac{1}{\tau!} $. Since $\tau>0$ in case (v),  we obtain that \eqref{e.cmcalt} equals 0.  It follows that   
\begin{eqnarray}
\sum_\text{case (v)}\cm(L, \tau, i,q,j, j') =  \sum_{  L+\tau\leq \rho, \tau>0}\,\sum_{(i,q,j)\in \ca_{L\tau}} 
\cm(L, \tau, i,q,j, j) = 0 . 
\label{e.cmi2}
\end{eqnarray}

\noindent\emph{Step 9: Conclusion.}  
In case (vi) it is easy to see that we have $L=1,\dots,[\rho]$, $\tau=q=j=j'=0$ and $i=L $. Precisely, we have 
\begin{eqnarray}
\text{case (vi)} = \{ ({L,0,L,0,0, 0}):    L=1,\dots, [\rho] \}. 
\label{e.casevi}
\end{eqnarray}  

Applying   \eqref{e.caseiii}, \eqref{e.youngf1}, \eqref{e.cmi2} and \eqref{e.casevi} to \eqref{e.caseall}  we obtain:   
\begin{eqnarray}
\lim_{n\to\infty}\cj (f, dx, \cp )  = \lim_{n\to\infty} \sum_{(L,\tau)\in\ci }\sum_{(i,q,j)\in\ca_{L\tau}}\sum_{j'=0}^{j}\cm({L,\tau,i,q,j, j'}) 
\notag
 \\
 = \lim_{|\cp|\to0}\lp\sum_{L=1}^{[\rho]}  \cm({L,0,L,0,0, 0}) +  \frac12 \cj  (f^{(1)}, dR, \cp) \rp. 
\label{e.cjfdx}
\end{eqnarray}

From \eqref{e.cmcajp} we have
\begin{eqnarray}
\sum_{L=1}^{[\rho]}\cm({L,0,L,0,0, 0})  = \sum_{L=1}^{[\rho]} \sum_{k=0}^{n-1}  \frac{1}{L!}\delta^{L}(f^{(L-1)}(t_{k},x_{t_{k}}) \be_{k}^{\otimes L} ) =  \cj  (f, \delta x, \cp)   . 
\label{e.cmzs1}
\end{eqnarray} 
Substituting \eqref{e.cmzs1} into \eqref{e.cjfdx} we obtain that relation \eqref{e.conv} holds. 

When $f^{(1)}$ is Lipschitz in $(t,x)$,  the fact that $x$ is of finite $p$-variation for $p>2\rho$ (see Remark \ref{remark.p}) implies that  $f^{(1)}(t,x_{t})$ is also of finite $p$-variation.  So  if $\rho<3/2$ or $\frac{1}{2\rho}+\frac{1}{\rho'}>1$ we have  that 
   $\cj  (f^{(1)}, dR, \cp) $ converges to the Young integral $\int_{0}^{T}f^{(1)}(t,x_{t})dR(t,t)$ as $|\cp|\to0$ (see Remark \ref{remark.young}). Applying this convergence to \eqref{e.conv} we obtain that relation \eqref{e.dxdelta} holds. 
The proof is now complete.
\end{proof}

 \section{ Multidimensional Skorohod-Stratonovich conversion formula}\label{section.md}
 
 In this section, we consider the Skorohod-Stratonovich conversion formula for a multidimensional Gaussian process.   In Section \ref{subsection.strat} we introduce a Stratonovich-type integral for continuous paths of finite $p$-variations for $p\geq 1$. Then in Section \ref{subsection.skorohod} we introduce the Skorohod-Riemann integral and prove Theorem \ref{thm.intro}.

\subsection{Stratonovich-type integral}\label{subsection.strat}
In the following  we introduce a Stratonovich-type integral.   For   a vector $i$ in $\bar{\NN}^{d}$  we   denote by  $i_{l}$ its $l$th component, and so we write $ i=(i_{1},\dots,i_{d})$. 
 For a real-valued function $ f (t,x)$, $t\in[0,T]$, $x=(x_{1},\dots, x_{d})\in\RR^{d}$ we denote the partial differential operators:
\begin{eqnarray}
\partial_{0} = \frac{\partial}{\partial t},
\qquad
\partial_{l} = \frac{\partial}{\partial x_{l}  }, \qquad l=1,\dots, d.
\notag
\end{eqnarray}

  \begin{Def}
  Let $x_{t}=(x_{1}(t),\dots, x_{d}(t)) $, $t\in[0,T]$ be a continuous path with values in $   \RR^{d}$, and denote   $x^{1,l}_{st}:=x_{l}(t)-x_{l}(s)$, $l=1,\dots, d$ and  $x^{1}_{st}:= x_{t}-x_{s}   $ for $s,t\in[0,T]$. Let $f$ be a continuous function on $[0,T]\times \RR^{d}$. Let $\cp$ be a partition of $[0,T]$: $0=t_{0}<\cdots<t_{n}=T$.   
Consider the compensated Riemann sum:
\begin{eqnarray}
\cj(\partial f, dx, \cp):=
\sum_{k=0}^{n-1} \sum_{m=1}^{\ell} \frac{1}{m!} \partial^{m}f(t_{k},x_{t_{k}}) (x_{t_{k}t_{k+1}}^{1})^{\otimes m}  ,  
\label{e.riemann}
\end{eqnarray}
where    $\ell\in\NN$ and we denote 
\begin{eqnarray}
\partial^{m}f(t_{k},x_{t_{k}}) (x_{t_{k}t_{k+1}}^{1})^{\otimes m} 
=
\sum_{j_{1},\dots, j_{m}=1}^{d}  \partial_{j_{1}}\cdots\partial_{j_{m}}f(t_{k},x_{t_{k}}) x^{1,j_{1}}_{t_{k}t_{k+1}}\cdots x^{1,j_{m}}_{t_{k}t_{k+1}} .
\label{e.pfx}
\end{eqnarray}
We define the  Stratonovich-type  integral:
\begin{eqnarray}
\int_{0}^{T} \partial f(t,x_{t}) dx_{t} := \lim_{|\cp|\to0} \cj(\partial f, dx, \cp), 
\label{e.riemannl}
\end{eqnarray}
given that the right-hand side of \eqref{e.riemann} is well-defined and that the  limit in \eqref{e.riemannl} exists.  

\end{Def}

\begin{remark}\label{remark.rough}
Let $(\bfx^{1}_{st},\dots, \bfx^{\ell}_{st})$, $s,t\in[0,T]$ be a geometric rough path of finite $p$-variation such that $\bfx^{1}_{st}=x_{t}-x_{s}$ (see e.g. \cite{FH20, FV10, lyons1998differential}).  
By some algebraic computations it can be shown that  
\begin{eqnarray}
\partial^{m}f(t_{k},x_{t_{k}})\cdot \bfx^{m}_{t_{k}t_{k+1}} = \frac{1}{m!}  \partial^{m}f(t_{k},x_{t_{k}})\cdot (x_{t_{k}t_{k+1}}^{1})^{\otimes m} , \qquad m=1,\dots, d,
\notag
\end{eqnarray}
where      the products in  both sides of \eqref{e.crsg} are     dot products. 
Recall that 
 the rough integral $\int_{0}^{T}\partial f (t,x_{t})d\bfx_{t}$ can be defined as: 
   \begin{eqnarray}
\int_{0}^{T}\partial f (t,x_{t})d\bfx_{t} = \lim_{|\cp|\to0}  \sum_{k=0}^{n-1} \sum_{m=1}^{[p]}   \partial^{m}f(t_{k},x_{t_{k}})\cdot \bfx^{m}_{t_{k}t_{k+1}}. 
\label{e.crsg}
\end{eqnarray}
 It follows that the rough integral $\int_{0}^{T}\partial f (t,x_{t})d\bfx_{t}$ coincides with the Stratonovich-type integral defined in \eqref{e.riemannl} given that both are well-defined. 
  \end{remark}

\begin{remark}\label{remark.jpfi}
For a vector $ i\in\bar{\NN}^{d}$ we   denote  $|i|= i_{1}+\cdots+i_{d}$ and  
\begin{eqnarray}
 \partial^{i}  =:\partial_{1}^{i_{1}}\cdots \partial_{d}^{i_{d}} =: 
\frac{\partial^{|i|}}{\partial x_{1}^{i_{1}}\cdots\partial x_{d}^{i_{d}}} . 
\notag
\end{eqnarray} 
From \eqref{e.riemann} it is easy to see that we have  
 \begin{eqnarray}
\cj (\partial f, dx , \cp ) =
\sum_{k=0}^{n-1}    \sum_{i\in  \bar{\NN}^{d }: 1\leq |i|\leq \ell}  
 \partial^{i}  f (t_{k}, x_{t_{k}})\cdot \lc \prod_{l=1}^{d}\frac{1}{i_{l}!} (x^{1,l}_{t_{k}t_{k+1}})^{i_{l}} \rc.
\label{e.rsstrat}
\end{eqnarray} 
\end{remark}
%

The following result shows that the  Stratonovich-type integral in \eqref{e.riemannl} is well-defined  given  proper regularity conditions of  $f$.

\begin{lemma}\label{lemma.strat}
Let $x_{t} $, $t\in[0,T]$ be a continuous path with values in $\RR^{d}$ such that it has finite $p$-variation for $p\geq 1$. Denote $\ell=[p]$. 
Let $f:(t,x)\in [0,T]\times \RR^{d} \to f(t,x)\in \RR$ be a continuous function such that it has bounded derivatives in $x$ up to order  $(\ell+1)$ and continuous first-order derivative in $t$.    Then the convergence \eqref{e.riemannl} holds, and we have the relation:
\begin{eqnarray}
\int_{0}^{T}\partial f(t,x_{t})d  x_{t}    = f(T,x_{T}) - f(0,x_{0}) - \int_{0}^{T}\partial_{0}f(u,x_{u})du.  
\label{e.roughint}
\end{eqnarray}

\end{lemma}
\begin{proof}
Take $s,t\in[0,T]$.  
Define the function 
$
\varphi(\la) = f(s , x_{s}+\la\cdot x^{1}_{st})$, 
 $ \la\in [0,1]. 
$ 
 By differentiating $\varphi$ for $m$ times we get:  
 \begin{eqnarray}
\varphi^{(m)}(\la) = \partial^{m} f (s , x_{s}+\la \cdot x^{1}_{st})\cdot ( x^{1}_{st})^{\otimes m},
\qquad 
m\in\NN,  
\label{e.phip}
\end{eqnarray}
where the right-hand side of \eqref{e.phip} should be interpreted as a dot product  similar to  \eqref{e.pfx}.
Applying the  Taylor expansion    we have
\begin{eqnarray}
\varphi(1)-\varphi(0) =\sum_{m=1}^{\ell}\frac{1}{m!} \varphi^{(m)}(0)    +\frac{1}{(\ell+1)!} \varphi^{(\ell+1)}(\theta)  ,
\label{e.ff}
\end{eqnarray}
where $\theta$ is a value in $[0,1]$. 
Note that the left-hand side of \eqref{e.ff} is just $f(s,x_{t})-f(s,x_{s})$. 
So relation \eqref{e.ff} together with \eqref{e.phip} gives:   
\begin{eqnarray}
f(t,x_{t}) - f(s,x_{s}) = [f(t,x_{t}) - f(s,x_{t})]+[f(s,x_{t}) - f(s,x_{s})]
\notag
\\
 = [f(t,x_{t}) - f(s,x_{t})]+ \sum_{m=1}^{\ell} \frac{1}{m!} \partial^{m}f(s,x_{s}) (x^{1}_{st})^{\otimes m}  
 \notag
 \\
 +\frac{1}{(\ell+1)!} \partial^{\ell+1}f(s,x_{s}+\theta\cdot x^{1}_{st}) (x^{1}_{st})^{\otimes (\ell+1)}. 
\label{e.ffst}
\end{eqnarray}
Consider a partition   $\cp$ of $[0,T]$: $0=t_{0}<\cdots<t_{n}=T$, and   write
\begin{eqnarray}
f(T,x_{T}) - f(0,x_{0}) = \sum_{k=0}^{n-1} [f(t_{k+1},x_{t_{k+1}}) - f(t_{k},x_{t_{k}})]. 
\label{e.ffot}
\end{eqnarray}
 Applying \eqref{e.ffst} to the right-hand side of \eqref{e.ffot} with $s=t_{k}$ and $t=t_{k+1}$ we get: 
 \begin{eqnarray}
f(T,x_{T}) - f(0,x_{0}) = \ca_{1}
+ \cj(\partial f, dx, \cp)
+ \ca_{2}\,,
\label{e.ffjj}
\end{eqnarray}
where
\begin{eqnarray}
\ca_{1} &=& \sum_{k=0}^{n-1} [f(t_{k+1},x_{t_{k+1}}) - f(t_{k},x_{t_{k+1}}) ] 
\notag
\\
\ca_{2}&=& \sum_{k=0}^{n-1} \frac{1}{(\ell+1)!} \partial^{\ell+1}f(t_{k},x_{t_{k}}+\theta_{k} \cdot x^{1}_{t_{k}t_{k+1}} ) (x^{1}_{t_{k}t_{k+1}}  )^{\otimes (\ell+1)}  ,
\notag
\end{eqnarray}
  $\theta_{k}$ is a value in $[0,1]$, and recall that $\cj(\partial f, dx, \cp)$ is defined in \eqref{e.riemann}. 
By mean value theorem  it is easy to see that $\ca_{1}$ is equal to a Riemann sum  of the integral $\int_{0}^{T}\partial_{0}f(t,x_{t})dt$. This implies that  
\begin{eqnarray}
\ca_{1}\to \int_{0}^{T}\partial_{0}f(t,x_{t})dt  
\label{e.j1cm}
\end{eqnarray}
 as $|\cp|\to 0$. 
On the other hand, since $\partial^{ \ell+1 }f $  is bounded  we have the estimate: 
\begin{eqnarray}
|\ca_{2}|\lesssim \sum_{k=0}^{n-1} |x^{1}_{t_{k}t_{k+1}}|^{\ell+1} . 
\notag
\end{eqnarray}
Invoking the assumption that  $x$ has finite $p$-variation and the fact that $\ell+1=[p]+1>p$, we can thus apply a similar argument as in \eqref{e.j1c} to show that  $\ca_{2}\to0$ as $|\cp|\to0$. 
Now, sending $|\cp|\to 0$ in \eqref{e.ffjj} and applying the convergences of $\ca_{1}$ in \eqref{e.j1cm} and $\ca_{2} $ we obtain: 
\begin{eqnarray}
f(T,x_{T}) - f(0,x_{0}) = \int_{0}^{T}\partial_{0}f(t,x_{t})dt 
+\lim_{|\cp|\to0} 
\cj(\partial f, dx, \cp).  
\notag
\end{eqnarray}
This shows that the limit of $\cj(\partial f, dx, \cp)$ exists, and it satisfies the relation \eqref{e.roughint}. The proof is now complete. 
\end{proof}



\subsection{Riemann-Skorohod integral and conversion formula}\label{subsection.skorohod}

Let $x_{l}$, $l=1,\dots,d$ be independent   Gaussian processes with covariance functions  $R_{l}(s,t)= \mE[x_{l}(s)x_{l}(t)]$. For each $x_{l}$ we can thus define the   Hilbert space  $\ch_{l}$ generated by $R_{l}$,  the Malliavin derivative operator $D_{l}$ and the divergence operators    $\delta_{l}$    as in Sections \ref{section.m}-\ref{section.cov}.  
In the following, we define the multidimensional version of the Skorohod-type integral:
\begin{Def}\label{def.skor}
Consider the following Skorohod-type Riemann sum: 
\begin{eqnarray}
\cj (\partial f, \delta x , \cp ): =
\sum_{k=0}^{n-1} \sum_{m=1}^{[\rho]} \frac{1}{m!}\sum_{j_{1},\dots, j_{m}=1}^{d} \delta_{j_{1}}\cdots\delta_{j_{m}}\Big( \partial_{j_{1}}\cdots\partial_{j_{m}}f(t_{k},x_{t_{k}}) \be_{k}^{\otimes m}
\Big) ,
\notag
\end{eqnarray}
where $\delta_{j_{1}}\cdots\delta_{j_{m}}$ is the divergence operator adjoint to the derivative operator  $D_{j_{1}}\cdots D_{j_{m}}$. 
As in Remark \ref{remark.jpfi}, we can show that  
\begin{eqnarray}
\cj (\partial f, \delta x , \cp ) =
\sum_{k=0}^{n-1}  
\,\sum_{i\in  \bar{\NN}^{d }:1\leq |i|\leq [\rho]}  
\lc\prod_{l=1}^{d}\frac{1}{i_{l}!} \rc\cdot 
\delta^{i} \Big( \partial^{i}  f (t_{k}, x_{t_{k}})\be_{k}^{\otimes |i|} \Big) ,
\notag
\end{eqnarray} 
where   $\delta^{i}  := \delta_{1}^{i_{1}}\cdots \delta_{d}^{i_{d}} $
   is the divergence operator adjoint to $ D^{i_{1}}_{1}\cdots D^{i_{d}}_{d}$. 
 We define the Skorohod-type integral as:
\begin{eqnarray}
\int_{0}^{T}\partial f(t,x_{t})dx_{t} := \lim_{|\cp|\to0} \cj (\partial f, \delta x , \cp ) , 
\label{e.skorl}
\end{eqnarray}
 given that $ \cj (\partial f, \delta x , \cp )$ is well-defined and that the limit exists in probability.   
\end{Def}

Following is our main result.  Recall that we denote $\partial^{2}_{ll} = \frac{\partial^{2}}{\partial x_{l}^{2}}$ and 
\begin{eqnarray}
\cj(\partial^{2}f, dR, \cp) := \sum_{k=0}^{n-1}  \sum_{l=1}^{d}  {\partial^{2}_{ll} f}  (t_{k},x_{t_{k}})\cdot \big(R_{l}(t_{k+1},t_{k+1}) - R_{l}(t_{k},t_{k}) \big).  
\label{e.fdr2}
\end{eqnarray}
     
\begin{theorem}\label{thm.main}
Let $ x $ be  a $\RR^{d}$-valued Gaussian process  satisfying Hypothesis \ref{hyp.1}. Let $\ell=[2\rho]$.  
Let  $f : (t,x)\in[0,T]\times \RR^{d}\to \RR$ be a continuous function with bounded derivatives in $x$ up to order $2\ell$ and continuous first-order derivative in $t$.   
Then the limit in  \eqref{e.riemannl} exists and we have the  relation:
\begin{eqnarray}
  \int_{0}^{T}\partial f(t,x_{t})d  x_{t}  =  \lim_{|\cp|\to0} \lp \cj(\partial f,\delta x, \cp) + \frac12 \cj(\partial^{2}f, dR, \cp) \rp. 
\label{e.pfdxpre}
\end{eqnarray}
Assume further that $\partial_{ll}^{2}f$, $l=1,\dots, d$ are Lipschitz in $(t,x)$, and that one of the following holds:

\noindent (1)  $\rho <3/2$; 
 
\noindent (2) $\frac{1}{2\rho}+\frac{1}{\rho'}>1$. 
 
\noindent Then   both  \eqref{e.skorl} and  \eqref{e.fdr2} converge as $|\cp|\to0$, and we have    the        conversion formula: 
 \begin{eqnarray}
  \int_{0}^{T}\partial f(t,x_{t})d  x_{t}  = \int_{0}^{T} \partial f(t,x_{t})\delta x_{t} +\frac12 \sum_{l=1}^{d}\int_{0}^{T} {\partial^{2}_{ll} f}  (t,x_{t})dR_{l}(t,t) , 
\label{e.dxdelta1}
\end{eqnarray}
where    $\int_{0}^{T} {\partial^{2}_{ll} f}  (t,x_{t})dR_{l}(t,t)$ should be interpreted as a Young integral.  
\end{theorem}
\begin{proof} 
We start by defining the following index sets: 
\begin{eqnarray}
\ca &=& \{(i,q,j)\in\bar{\NN}^{3d}: 1\leq |i|\leq \ell,\,   2q  \leq i ,   \,   j  \leq i -2q \},  
\notag
\\
\ca_{L } &=& \{(i,q,j)\in \ca:   i-2q-j=L \}, 
\qquad \qquad\text{for}\qquad  L\in\bar{\NN}^{d}:  |L|\leq \ell, 
\notag
\\
\ca_{L\tau} &=& \{(i,q,j)\in \ca_{L}: q+j=\tau \} , 
\qquad \qquad\text{for}  \qquad \tau\in\bar{\NN}^{d}:   |\tau|\leq    \ell-|L|,
\notag
\\
\ci &=& \{(L, \tau)\in\bar{\NN}^{2d}:   |L|\leq \ell,   |\tau|\leq \ell-|L|\} .
\notag
\end{eqnarray}
In the above the inequality $2q\leq i$ means that $2q_{l}\leq i_{l}$ for all $l=1,\dots,d$. The other vector inequalities should be  interpreted in the similar way. As in \eqref{e.cal} it is easy to see that we have
\begin{eqnarray}
\ca_{L} = \{(i,q,j)\in \bar{\NN}^{3d}: 1\leq |i|\leq \ell,\,   i-2q-j=L \}, 
\notag
\end{eqnarray}
and therefore
  \begin{eqnarray}
\ca_{L\tau} = \{(i,q,j)\in \bar{\NN}^{3d}: 1\leq |i|\leq \ell,\,   i-2q-j=L , q+j=\tau\}.  
\notag
\end{eqnarray}
It follows that if $(L,\tau)$ is such that $|L+2\tau|\leq \ell$, then we have 
\begin{eqnarray}
\ca_{L\tau} = \Big\{(i,q,j) : (i_{l},q_{l},j_{l})\in\big\{(L_{l}+\tau_{l},0,\tau_{l}),\dots, (L_{l}+2\tau_{l},\tau_{l},0)\big\},l=1,\dots, d \Big\} . 
\label{e.caltm}
\end{eqnarray}

As in   Theorem \ref{thm.dxdelta} by some careful algebraic computations     we can show the following decomposition: 
 
\begin{eqnarray}
\cj (f, dx, \cp )  =\sum_{(L,\tau)\in\ci} \sum_{(i,q,j)\in\ca_{L\tau}}\sum_{0\leq j'\leq j} \cm(L,\tau,i,q,j,j') , 
\notag
\end{eqnarray}
where
\begin{eqnarray}
\cm(L,\tau,i,q,j,j')  = \sum_{k=0}^{n-1} \lc \prod_{l=1}^{d} \frac{1}{2^{q_{l}}q_{l}!j_{l}!L_{l}!} \rc\cdot   \delta^{L}  \Big( \partial^{L+2\tau}f (t_{k}, x_{t_{k}}) \cdot  \be_{k}^{\otimes |L|} \Big) 
 \notag\\\cdot
\lc \prod_{l=1}^{d}\si_{l,k}^{2(j_{l}'+q_{l})} \rc \cdot \lc \prod_{l=1}^{d} \frac{1}{2^{j_{l}}}
 {j_{l}\choose j_{l}'} (-1)^{j_{l}'}(R_{l}(t_{k+1},t_{k+1}) - R_{l}(t_{k},t_{k})  )^{j_{l}-j_{l}'}\rc,   
\label{e.cmm}
\end{eqnarray}
       $\si_{l,k}=\mE[(x^{1,l}_{t_{k}t_{k+1}})^{2}]^{1/2}$ and    $\delta^{L}:= \delta^{L_{1}}_{1}\cdots \delta^{L_{d}}_{d}$.
 
Now we   estimate $\cm(L,\tau,i,q,j,j') $ in various  cases as in Theorem \ref{thm.dxdelta}. Precisely, we consider the following six cases:

\noindent(i) \quad $ |j-j'|\geq2$;

\noindent(ii) \quad $ |j-j'|=1 $, $|L+\tau| >1$;

\noindent(iii) \quad $ j=j'$, $|L+\tau|>\rho$.

\noindent(iv) \quad $|j-j'|=1$, $|L+\tau|=1$;

\noindent(v) \quad $j=j'$, $|L+\tau| \leq \rho$, $|\tau| >0$;

\noindent(vi) \quad $j=j'$, $|L|  \leq \rho$, $\tau =0$.

Similar to the proof of Theorem \ref{thm.dxdelta}, given $\frac{1}{\rho}+\frac{1}{\rho'}>1$ it can be shown that 
\begin{eqnarray}
\cm(L,\tau,i,q,j,j') \to0 
\label{e.cmcasei}
\end{eqnarray}
 as $|\cp|\to0$ in cases (i)-(iii). For sake of conciseness,  we omit the proof  and leave it to the patient reader.

In case (v) we have $j-j'=0$ and thus $ j'+q=j+q=\tau$. Then from \eqref{e.cmm} we have   
\begin{eqnarray}
\sum_{(i,q,j)\in\ca_{L\tau}}
\cm(L,\tau,i,q,j,j) =  \sum_{(i,q,j)\in\ca_{L\tau}}\sum_{k=0}^{n-1}     \delta^{L } \Big( \partial^{L+2\tau}f (t_{k}, x_{t_{k}}) \cdot \be_{k}^{\otimes |L|} \Big)
 \notag\\
\cdot\lc \prod_{l=1}^{d}\frac{1}{L_{l}!}\si_{l,k}^{2\tau_{l}} \rc \cdot \lc \prod_{l=1}^{d}  
     \frac{(-1)^{j_{l}}}{2^{j_{l}}2^{q_{l}}q_{l}!j_{l}!}  \rc.
\end{eqnarray}
As in the proof of Theorem \ref{thm.dxdelta}, since $|L+\tau| \leq \rho$ we can show that $|L+2\tau|\leq [2\rho]= \ell$, and therefore   \eqref{e.caltm} holds. 
 So we can write  
\begin{eqnarray}
\sum_{(i,q,j)\in\ca_{L\tau}}
\cm(L,\tau,i,q,j,j) =   \sum_{k=0}^{n-1}   \delta^{L } \Big( \partial^{L+2\tau}f (t_{k}, x_{t_{k}}) \cdot \be_{k}^{\otimes |L|} \Big) 
\cdot\lc \prod_{l=1}^{d}\frac{1}{L_{l}!}\si_{l,k}^{2\tau_{l}} \rc\notag
\\
 \cdot  \prod_{l=1}^{d}  
  \lc  \sum_{q_{l}+j_{l}=\tau_{l}}  \frac{(-1)^{j_{l}}}{2^{j_{l}}2^{q_{l}}q_{l}!j_{l}!}  \rc  . 
\notag
\end{eqnarray}
Since $|\tau|>0$ the above shows that  in case (v) we have  
\begin{eqnarray}
\sum_{(i,q,j)\in\ca_{L\tau}}
\cm(L,\tau,i,q,j,j) =0. 
\label{e.cmcasev}
\end{eqnarray}

In case (iv) we must have $|j|=1 $, $q=j'=L=0$ and $j=\tau=i$. This implies that
\begin{eqnarray}
\text{case (iv)} =  \frac12 \cj(\partial^{2}f, dR, \cp). 
\label{e.youngr}
\end{eqnarray}

Finally, it is easy to verify that 
\begin{eqnarray}
\text{case (vi)} = \cj(\partial f, \delta x, \cp). 
\label{e.cmcaseiv}
\end{eqnarray}

In summary of the convergence \eqref{e.cmcasei} in case (i)-(iii), relation \eqref{e.cmcasev}, \eqref{e.youngr} and \eqref{e.cmcaseiv}  we conclude that \eqref{e.pfdxpre} holds. 

As in Theorem \ref{thm.dxdelta} we can show that if $\partial_{ll}^{2}f$, $l=1,\dots, d$ are Lipschitz in $(t,x)$ and  if $\rho <3/2$ or $\frac{1}{2\rho}+\frac{1}{\rho'}>1$, then $\cj(\partial^{2}f, dR, \cp)  $ converges to the Young integral $ \sum_{l=1}^{d}\int_{0}^{T} {\partial^{2}_{ll} f}  (t,x_{t})dR_{l}(t,t) $ as $|\cp|\to0$. 
 This implies that \eqref{e.dxdelta1} holds.  The proof is now complete. 
\end{proof}

\bibliographystyle{abbrv}
\bibliography{Skorohod-Riemann.bib} 


\end{document}